\documentclass[final,onefignum,onetabnum]{siamonline250211}
\newif\ifcombinedpaper\combinedpapertrue
\newif\ifsupplementpaper\supplementpaperfalse
\usepackage{amssymb,mathtools}
\usepackage{microtype}
\usepackage{enumitem}
\usepackage{booktabs}
\usepackage{graphicx}
\usepackage{placeins}
\usepackage{float}
\graphicspath{{figures/}}
\hypersetup{pdftitle={Minimax Rates for Conditional Generative Learning under Wasserstein Loss},pdfauthor={Sungwon Park, Kyoung-Sook Moon, and Hongjoong Kim}}
\allowdisplaybreaks
\newsiamthm{assumption}{Assumption}
\newsiamremark{remark}{Remark}

\newcommand{\X}{\mathcal X}
\newcommand{\Y}{\mathcal Y}
\newcommand{\Pcal}{\mathcal P}
\newcommand{\Qcal}{\mathcal Q}
\newcommand{\E}{\mathbb E}
\newcommand{\1}{\mathbf 1}

\newcommand{\TV}{\operatorname{TV}}
\newcommand{\Id}{\operatorname{Id}}
\newcommand{\KR}{\mathrm{KR}}

\headers{Minimax Conditional Generative Learning}{S. Park, K.-S. Moon, and H. Kim}
\title{Minimax Rates for Conditional Generative Learning under Wasserstein Loss}
\author{
{Sungwon Park}\thanks{Department of Mathematics, Korea University
(\email{sungwpark@korea.ac.kr},
 \email{hongjoong@korea.ac.kr}).}
\and
{Kyoung-Sook Moon}\thanks{Department of Finance and Big Data, Gachon University
(\email{ksmoon@gachon.ac.kr}).}
\and
{Hongjoong Kim}\footnotemark[1]
}
\date{}

\newcommand{\AppendixRef}[1]{%
  \ifcombinedpaper Appendix~\ref{#1}%
  \else\ifsupplementpaper Section~\ref{#1}%
  \else Supplementary Material, Section~\ref{#1}%
  \fi\fi}

\begin{document}
\maketitle

\begin{abstract}
	We study minimax rates for conditional generative learning under integrated Wasserstein loss, with H\"older continuous conditional laws and an unrestricted covariate distribution. We evaluate unrestricted generators through their conditional output distributions. For expected Wasserstein loss, the minimax rate agrees with the rate under uniform covariates and follows from balancing localization bias against unconditional distribution estimation error. Higher moments of the loss can instead be dominated by unobserved covariate regions, which force an additional risk contribution of order $n^{-1}$, where $n$ is the sample size. An adaptive localization argument gives matching upper bounds. The results cover unrestricted response laws and possibly vanishing Besov densities for $W_1$, and extend to higher Wasserstein orders under uniform response density bounds and integrability conditions. A multiscale lower bound closes the logarithmic gap in unconditional $W_1$ minimax rates in dimension two for both unrestricted distributions and Besov density classes. We also establish minimax rates for conditional Brenier maps under density, integrability, and uniform curvature conditions.
\end{abstract}



\section{Introduction}\label{sec:intro}

Many prediction problems require a conditional distribution rather than a conditional mean, for example when the response is intrinsically random or multimodal. Such problems arise in probabilistic forecasting \cite{RaspLerch2018,SalinasEtAl2020}, simulation-based Bayesian inference \cite{PapamakariosMurray2016,GreenbergEtAl2019}, and probabilistic surrogate modeling for scientific systems \cite{BulteSchollKutyniok2025,ParkEtAl2026}. We formulate conditional generative learning as nonparametric estimation of conditional output distributions and study its minimax rates under Wasserstein loss.

Let $(X_1,Y_1),\ldots,(X_n,Y_n)$ be i.i.d. observations with $X\in[0,1]^{d_x}$ and $Y\in[0,1]^{d_y}$, and let
$$
	x\longmapsto Q_x=P_{Y\mid X=x}
$$
be a regular conditional distribution of $Y$ given $X=x$. We integrate $W_p^\kappa(\widehat Q_x,Q_x)$ with respect to the marginal law $P_X$. We impose no density, doubling, or lower mass condition on $P_X$, which may be discrete, continuous, singular, or highly nonuniform. For $p=1$ we allow unrestricted distributions and densities in Besov classes that need not be bounded away from zero. For $1<p<\infty$ we assume that the response densities are uniformly bounded away from zero and infinity, while conditional Brenier maps are treated under a separate geometric stability class.

For the uniform covariate distribution, the benchmark rate follows from the usual local bias--variance balance. Pooling observations over a covariate neighborhood of diameter $h$ incurs H\"older error of order $h^\beta$, while the effective local sample size is $nh^{d_x}$. The stochastic term is therefore the minimax error for estimating one distribution from $nh^{d_x}$ observations. Balancing these terms defines $A_n(\tau)$, with $\tau=0$ for unrestricted distributions and $\tau>0$ for densities in a Besov class.

Our first result shows that $A_n(\tau)$ remains sharp for integrated Wasserstein loss over all Borel covariate marginals. Weighting local errors by cell probability allows them to be summed uniformly over $P_X$. Higher moments behave differently: an atom of mass of order $n^{-1}$ is unobserved with probability bounded away from zero, forcing an additional $n^{-1}$ contribution to the integrated $W_p^\kappa$ risk. The resulting rate is $A_n(\tau)^\kappa\vee n^{-1}$.

For $p>1$, a comparison with a negative Sobolev norm gives the local distribution estimation rate under uniform response density bounds. The corresponding bounds extend to every fixed finite moment. The choices $\kappa=1$ and $\kappa=p$ give Wasserstein loss and its $p$th power, respectively.

The conditional results rely on sharp unconditional distribution estimation rates. We prove the missing $W_1$ lower bound of order $\sqrt{\log m/m}$ for unrestricted distributions and Besov density classes in dimension two. A single resolution gives only $m^{-1/2}$; the logarithmic factor comes from combining perturbations across a logarithmic number of admissible resolutions.

A conditional generator transforms a reference noise law into a conditional response law. We also study conditional Brenier maps, the quadratic optimal transport maps from normalized Lebesgue measure $\rho$ on the response cube. Uniform curvature bounds transfer squared $W_2$ error for the target distribution to squared $L^2(\rho)$ map error. A construction based on inverse potentials gives the matching lower bound, yielding the minimax map rate $A_n(\tau)^2$.

\subsection{Related work}

Existing minimax results for conditional density and distribution estimation use several notions of loss and regularity. Efromovich \cite{Efromovich2007} develops minimax theory for conditional density estimation under fixed and random designs, and Brunel, Comte, and Lacour \cite{BrunelComteLacour2010} study conditional distribution functions. Li, Neykov, and Balakrishnan \cite{LiNeykovBalakrishnan2022} obtain minimax rates without a marginal density assumption under total variation regularity in the covariate, while Bilodeau, Foster, and Roy \cite{BilodeauFosterRoy2023} derive rates from empirical Hellinger entropy. This setting is also related to distribution-free nonparametric regression \cite{GyörfiKohlerKrzyzakWalk2002}; for example, Kohler and Krzy\.zak \cite{KohlerKrzyzak2013} show that irregular random designs can create an $n^{-1}$ barrier in a noiseless interpolation problem.

The closest Wasserstein results differ mainly in their assumptions on the covariate and response distributions. Dombry, Modeste, and Pic \cite{DombryModestePic2025} prove universal consistency in $W_p$ and obtain $W_1$ minimax rates for univariate responses on particular classes without restricting the covariate marginal. Pic et al. \cite{PicEtAl2023} study minimax rates under the CRPS. For multivariate responses, B\'en\'ezet, Cheng, and Jaimungal \cite{BenezetChengJaimungal2025} analyze local empirical estimators under a uniform lower bound on the covariate law relative to Lebesgue measure. Tang and Yang \cite{TangYang2025} obtain minimax rates, up to logarithmic factors, under H\"older integral probability metrics for smooth density models and distributions supported on manifolds. We impose H\"older continuity directly on $x\mapsto Q_x$ and obtain matching multivariate Wasserstein minimax bounds over all Borel covariate marginals.

Wasserstein-based conditional generators and conditional diffusion models are studied in \cite{LiuEtAl2021,TangLinYang2025,SongEtAl2026}. Kumar, Yang, and Lin \cite{KumarYangLin2025} analyze likelihood-based conditional generative models. Chemseddine et al. \cite{ChemseddineEtAl2025} study conditional Wasserstein distances. Tamano and Imaizumi \cite{TamanoImaizumi2026} derive minimax bounds for causal conditional distribution learning under a regular target design and a smooth latent pushforward model. Balakrishnan \cite{Balakrishnan2026} shows that estimating a valid transport map includes estimating its induced distribution. Our results give analogous statistical limits for conditional generators on the H\"older classes considered here.

For the unconditional distribution estimation problem we use empirical Wasserstein bounds \cite{FournierGuillin2015,SinghPoczos2018,WeedBach2019,BobkovLedoux2021} and results for smooth densities under Wasserstein and related adversarial losses \cite{SinghEtAl2018,UppalSinghPoczos2019,Liang2021,NilesWeedBerthet2022,Divol2022}. Dedecker and Merlev\`ede \cite{DedeckerMerlevede2017} study higher moments of empirical $W_1$ errors, and Bobkov and Ledoux \cite{BobkovLedoux2021} give a Fourier analytic proof of the critical two-dimensional empirical rate. Niles-Weed and Berthet \cite{NilesWeedBerthet2022} obtain the optimal polynomial rates for smooth densities and identify the role of a density lower bound for $p>1$. Their lower bound leaves a logarithmic gap in the critical two-dimensional case. Divol \cite{Divol2022} gives the corresponding critical upper rate for densities on smooth manifolds without boundary. Section~\ref{sec:response} establishes the matching critical lower bound for unrestricted distributions and the Besov classes considered here.

The Brenier map result builds on minimax and plugin estimation of smooth transport maps \cite{HutterRigollet2021,ManoleEtAl2024,DivolNilesWeedPooladian2025,VacherEtAl2024} and on recent stability and minimax comparison results \cite{BalakrishnanManole2026,Balakrishnan2026}. Covariate-dependent constructions appear in \cite{TabakTrigilaZhao2020,TabakTrigilaZhao2021,HosseiniHsuTaghvaei2025}, including the conditional Brenier estimator of Baptista et al. \cite{BaptistaEtAl2025} and conditional map error bounds of Al-Jarrah et al. \cite{AlJarrahEtAl2025Error}. Combining these stability results with our localization bounds gives an upper bound uniform over all $P_X$. The inverse potential construction gives the matching conditional lower bound in the stated class.

Section~\ref{sec:model-main} states the model and main results. Section~\ref{sec:conditionalization-main} proves the localization argument for unrestricted covariate distributions, Section~\ref{sec:response} gives the multiscale response lower bound, and Section~\ref{sec:brenier-main} treats conditional Brenier maps. \ifcombinedpaper The appendices contain\else The supplementary material contains\fi{} the remaining unconditional estimates, adaptive localization, bounds for higher Wasserstein orders, measurability arguments, and class verifications.

\section{Main results}\label{sec:model-main}

We define the losses and regularity classes, then state the minimax rates for conditional distributions and Brenier maps.

\subsection{Statistical model and losses}\label{subsec:model-loss}

Let $d_x,d_y\ge1$, $\X=[0,1]^{d_x}$, and $\Y=[0,1]^{d_y}$. Suppose that $(X_1,Y_1),\ldots,(X_n,Y_n)$ are i.i.d. with joint law $P$.
We denote the marginal law of $X$ by $P_X$, write $S_X=\operatorname{supp}(P_X)$, and fix a Borel regular conditional distribution
$$
	Q_x=P_{Y\mid X=x}
$$
of $Y$ given $X=x$.

For probability measures $\mu,\nu$ on $[0,1]^d$, let $\Pi(\mu,\nu)$ denote the probability measures on $[0,1]^d\times[0,1]^d$ with marginals $\mu$ and $\nu$. For $1\le p<\infty$, define
\begin{equation}\label{eq:wasserstein-intro}
	W_p^p(\mu,\nu)
	=\inf_{\pi\in\Pi(\mu,\nu)}
	\int |y-y'|^p\,\pi(dy,dy').
\end{equation}
Thus $W_p^p(\mu,\nu)$ is the minimum expected transportation cost over all couplings of $\mu$ and $\nu$. For $p=1$, the Kantorovich--Rubinstein theorem gives
\begin{equation}\label{eq:kr-duality-prelim}
	W_1(\mu,\nu)
	=\sup_{\operatorname{Lip}(f)\le1}
	\left|\int f\,d(\mu-\nu)\right|.
\end{equation}
For a finite signed measure $\sigma$ with total mass zero, we similarly set
\begin{equation}\label{eq:kr-norm-prelim}
	\|\sigma\|_{\KR}
	=\sup_{\operatorname{Lip}(f)\le1}
	\left|\int f\,d\sigma\right|.
\end{equation}
Hence $W_1(\mu,\nu)=\|\mu-\nu\|_{\KR}$. We use this dual representation in the unconditional lower bound of Section~\ref{sec:response}.

For $1\le p<\infty$, $1\le\kappa<\infty$, and an estimator $x\mapsto\widehat Q_x$, define the integrated $W_p^\kappa$ risk by
\begin{equation}\label{eq:risk}
	\mathcal R_{n,p,\kappa}(\widehat Q;P)
	=\E_{P^{\otimes n}}\!\left[
	\int_\X W_p^\kappa(\widehat Q_x,Q_x)\,P_X(dx)
	\right].
\end{equation}
For a class $\Pcal$ of joint distributions, the corresponding minimax risk is
\begin{equation}\label{eq:minimax}
	\mathcal R^\star_{n,p,\kappa}(\Pcal)
	=\inf_{\widehat Q}\sup_{P\in\Pcal}
	\mathcal R_{n,p,\kappa}(\widehat Q;P).
\end{equation}
The $\kappa$th root of the risk in \eqref{eq:risk} is the $L^\kappa$ norm of the Wasserstein error over the training sample and an independent test covariate. The infimum is over probability kernels measurable jointly in the training sample and the query covariate.

A measurable conditional generator driven by independent noise induces such a kernel. Conversely, since $\Y$ is a standard Borel space, every such kernel can be realized by a measurable function of the training sample, the query covariate, and an independent uniform random variable \cite[Lemma~3.22]{Kallenberg2002}. Thus conditional distribution estimation and conditional generative learning have the same minimax risk when generators are unrestricted and evaluated through their output laws.

For the integrated $W_p^\kappa$ risk, the cases $\kappa=1$ and $\kappa=p$ give ordinary $W_p$ loss and the usual $W_p^p$ transport cost, respectively. For $p=1$ the results hold on the broad distribution classes introduced below, whereas for $1<p<\infty$ they use the density bounds in Assumption~\ref{ass:wp-class}. All regularity assumptions on $x\mapsto Q_x$ are imposed only on $S_X$, since values of a conditional version outside the support do not affect \eqref{eq:risk}.

For conditional map estimation, let $\rho$ be normalized Lebesgue measure on $\Y$. For a measurable map $T:\Y\to\Y$, $T_\#\rho$ denotes the law of $T(Z)$ for $Z\sim\rho$. Since $\rho$ is absolutely continuous, the quadratic optimal transport map $T_x=\nabla\varphi_x$ from $\rho$ to $Q_x$ exists and is unique $\rho$-almost everywhere, with $\varphi_x$ convex \cite{Brenier1991,Santambrogio2015}. Map estimators $\widehat T_x:\Y\to\Y$ are based on the training sample and are required to be measurable jointly in the sample and covariate as $L^2(\rho)$-valued functions. Define their integrated squared $L^2(\rho)$ risk by
\begin{equation}\label{eq:map-risk-main}
	\mathcal R_n^{\mathrm{map}}(\widehat T;P)
	=\E_{P^{\otimes n}}\!\left[\int_\X
	\|\widehat T_x-T_x\|_{L^2(\rho)}^2\,P_X(dx)\right].
\end{equation}

We use $a\vee b=\max\{a,b\}$ and $t_+=\max\{t,0\}$. We write $a\lesssim b$ when $a\le Cb$ for a constant $C$ independent of both the sample size and the covariate distribution, and $a\asymp b$ when $a\lesssim b$ and $b\lesssim a$. The dimensions and model parameters are fixed in their stated ranges. Implicit constants and sample-size thresholds may depend on them, on each fixed loss exponent $\kappa$, and on the chosen wavelet system.

\subsection{Regularity classes}\label{subsec:regularity}

The assumptions separate regularity in the covariate from regularity of the response law. In all four classes below, $P_X$ ranges over all Borel probability measures on $\X$. Fix $0<\beta\le1$ and $0<L_X<\infty$. For a specified Wasserstein order $1\le p<\infty$, we use the covariate H\"older condition
\begin{equation}\label{eq:cond-holder}
	W_p(Q_x,Q_{x'})\le L_X|x-x'|^\beta,
	\qquad x,x'\in S_X.
\end{equation}
Each assumption below specifies the order $p$ and any additional regularity of the response law.

\begin{assumption}[Unrestricted response $W_1$ model]\label{ass:free-class}
	There is a Borel version $x\mapsto Q_x$ of the conditional distribution for which \eqref{eq:cond-holder} holds with $p=1$. No density or response smoothness condition is imposed on the individual distributions $Q_x$.
\end{assumption}
Denote this class by $\Pcal^{\mathrm{free}}_\beta$.

For densities with smoothness $\tau>0$, let $1\le r,v\le\infty$ and use the standard Besov space $B^\tau_{r,v}([0,1]^d)$; see \cite{Cohen2003,TriebelDomains2008}. Here $\tau$ is the smoothness index in the response variable, while $r$ and $v$ determine aggregation within and across resolutions. We use the equivalent characterization by boundary wavelets recalled in \AppendixRef{supp:guide}. For each response dimension $d$, fix $L_Y>\|1\|_{B^\tau_{r,v}([0,1]^d)}$, taking $d=d_y$ in the conditional models.

\begin{assumption}[Besov density model for $W_1$]\label{ass:class}
	There is a jointly Borel function $(x,y)\mapsto q_x(y)$ such that $Q_x=q_x\,dy$ for $x\in S_X$, \eqref{eq:cond-holder} holds with $p=1$, and
	\begin{equation}\label{eq:besov-bound}
		\sup_{x\in S_X}\|q_x\|_{B^\tau_{r,v}(\Y)}\le L_Y.
	\end{equation}
	No pointwise lower or upper bound on $q_x$ is assumed.
\end{assumption}
Denote the resulting class by $\Pcal^B_{\beta,\tau,r,v}$. For the associated unconditional problems, let $\Qcal_d^{\mathrm{free}}$ be the class of all Borel probability measures on $[0,1]^d$ and put
\begin{equation}\label{eq:smooth-response-class}
	\Qcal^B_{d,\tau,r,v}
	=\left\{q\,dy:q\ge0,\ \int q=1,\ \|q\|_{B^\tau_{r,v}}\le L_Y\right\}.
\end{equation}
For higher Wasserstein orders we assume that the conditional densities are uniformly bounded away from zero and infinity. Fix $1<p<\infty$, $\tau\ge0$, $1\le r,v\le\infty$, and constants $0<\underline q<1<\overline q<\infty$.

\begin{assumption}[$W_p$ density bounds]\label{ass:wp-class}
	There is a jointly Borel function $(x,y)\mapsto q_x(y)$ such that $Q_x=q_x\,dy$ for $x\in S_X$, \eqref{eq:cond-holder} holds for the fixed order $1<p<\infty$, and
	\begin{equation}\label{eq:wp-density-bounds}
		\underline q\le q_x(y)\le\overline q,
		\qquad x\in S_X,\quad\text{for a.e. }y\in\Y.
	\end{equation}
	When $\tau>0$, we also impose \eqref{eq:besov-bound} and require $r\ge p$ if $d_y>2$; when $\tau=0$, no Besov condition is imposed and the indices $r,v$ are ignored.
\end{assumption}
Denote this class by $\Pcal^{(p)}_{\beta,\tau,r,v}(\underline q,\overline q)$, and let $\Qcal^{(p)}_{d,\tau,r,v}(\underline q,\overline q)$ be the corresponding unconditional distribution class on $[0,1]^d$, with $r\ge p$ when $\tau>0$ and $d>2$. The lower density bound is used in the comparison with a negative Sobolev norm for $W_p$, and the upper bound controls the stochastic wavelet error. These conditions are sufficient for the argument below and are not claimed to be minimal.

For the conditional Brenier map class, fix $\tau\ge0$, $1\le r,v\le\infty$, $1<\overline q<\infty$, and $\Lambda>1$.

\begin{assumption}[Brenier map regularity]\label{ass:brenier-class}
	There is a jointly Borel function $(x,y)\mapsto q_x(y)$ such that $Q_x=q_x\,dy$ for every $x\in S_X$, \eqref{eq:cond-holder} holds with $p=2$, and
	\begin{equation}\label{eq:brenier-class-main}
		0<q_x(y)\le\overline q,
		\qquad x\in S_X,\quad\text{for a.e. }y\in\Y.
	\end{equation}
	When $\tau>0$, we also impose \eqref{eq:besov-bound} and require $r\ge2$ if $d_y>2$; when $\tau=0$, the Besov indices are ignored. In addition, $\varphi_x\in C^{1,1}(\Y)$ and $\operatorname{Lip}(\nabla\varphi_x)\le\Lambda$ uniformly in $x\in S_X$.
\end{assumption}
Denote the resulting class by $\Pcal^{\mathrm{Br}}_{\beta,\tau,r,v}(\overline q,\Lambda)$.

The parameters $\beta$ and $\tau$ play different roles. The former controls the approximation error from pooling nearby covariates, whereas the latter determines the difficulty of estimating a response distribution from a local sample.

\subsection{Conditional distribution estimation}

We begin with the unconditional distribution estimation problem. Let $a_{d,\tau}(m)$ denote the minimax $W_1$ rate for estimating a distribution on $[0,1]^d$ from $m$ observations, with $\tau=0$ for unrestricted laws and $\tau>0$ for the Besov class above. When $\tau=0$, this notation is only a convention for the unrestricted response model and does not impose a $B^0_{r,v}$ condition. We use the following expression, extended to real $m\ge1$ for localization:
\begin{equation}\label{eq:response-root-intro}
	a_{d,\tau}(m)=
	\begin{cases}
		m^{-1/2}, & d=1,\\[1mm]
		\{\log(em)/m\}^{1/2}, & d=2,\\[1mm]
		m^{-(\tau+1)/(2\tau+d)}, & d>2.
	\end{cases}
\end{equation}
The rate is parametric in dimension one, has the critical logarithmic correction in dimension two, and improves with smoothness in the response variable when $d>2$. We use this benchmark throughout the conditional analysis.

Except for the critical two-dimensional lower bound, these rates follow from existing results on empirical Wasserstein convergence and smooth density estimation. Theorem~\ref{thm:response-benchmark} in Section~\ref{sec:response} states the unconditional benchmark and establishes the critical two-dimensional lower bound; \AppendixRef{supp:w1-upper} extends the first-moment bounds to every fixed finite $L^\kappa$ risk.

Under the uniform covariate distribution, a cell of side length $h$ has expected sample count $nh^{d_x}$. Pooling within the cell incurs H\"older approximation error of order $h^\beta$, while Theorem~\ref{thm:response-benchmark} gives distribution estimation error $a_{d_y,\tau}(nh^{d_x})$. We define $A_n(\tau)$ as the common order obtained by balancing
\begin{equation}\label{eq:balancing-main}
	h^\beta\asymp a_{d_y,\tau}(nh^{d_x}).
\end{equation}
Writing
\begin{equation}\label{eq:gamma-main}
	\gamma(\tau)
	=\left\{2+\frac{d_x}{\beta}
	+\frac{(d_y-2)_+}{\tau+1}\right\}^{-1},
\end{equation}
we have
\begin{equation}\label{eq:An-main}
	A_n(\tau)=
	\begin{cases}
		n^{-\gamma(\tau)}, & d_y\ne2,\\[1mm]
		\{\log(en)/n\}^{\gamma(\tau)}, & d_y=2.
	\end{cases}
\end{equation}
For unrestricted responses, we use these definitions with $\tau=0$.

With an unrestricted covariate distribution, the localization scale remains $A_n(\tau)$, but higher moments introduce an additional term.

\begin{theorem}[Conditional $W_1$ minimax rates]\label{thm:main}\label{thm:higher-moments}
	Consider the classes defined by Assumptions~\ref{ass:free-class} and~\ref{ass:class}, respectively, with fixed model parameters. Fix $1\le\kappa<\infty$. Then, for all sufficiently large $n$,
	\begin{align}
		\mathcal R^\star_{n,1,\kappa}(\Pcal^{\mathrm{free}}_\beta)
		&\asymp A_n(0)^\kappa\vee n^{-1},\label{eq:free-higher-rate}\\*
		\mathcal R^\star_{n,1,\kappa}(\Pcal^B_{\beta,\tau,r,v})
		&\asymp A_n(\tau)^\kappa\vee n^{-1},
		\qquad \tau>0.\label{eq:higher-rate-w1}
	\end{align}
\end{theorem}

The first term is necessary already under the uniform covariate distribution, whereas a two-point covariate distribution with an atom of mass $n^{-1}$ forces the second term. For $\kappa=1$, the additional $n^{-1}$ term is of lower order and the theorem gives the ordinary integrated $W_1$ rate $A_n(\tau)$.

Writing
\begin{equation}\label{eq:kappa-critical}
	\kappa_c(\tau)=\gamma(\tau)^{-1}
	=2+\frac{d_x}{\beta}+\frac{(d_y-2)_+}{\tau+1},
\end{equation}
the term $A_n(\tau)^\kappa$ dominates the minimax risk for $\kappa<\kappa_c(\tau)$, while $n^{-1}$ dominates for $\kappa>\kappa_c(\tau)$. The polynomial orders agree at the boundary. When $d_y=2$, the logarithmic correction gives $A_n(\tau)^{\kappa_c}=\log(en)/n$.

\begin{table}[ht]
	\centering
	\caption{Unconditional and conditional rates. The response dimension determines the critical behavior, while the covariate dimension enters through localization.}\label{tab:phase-diagram}
	\begin{tabular}{@{}lll@{}}
		\toprule
		Regime & Unconditional rate $a_{d_y,\tau}(m)$ & Conditional rate $A_n(\tau)$\\
		\midrule
		$d_y=1$ & $m^{-1/2}$ & $n^{-\beta/(2\beta+d_x)}$\\[1mm]
		$d_y=2$ & $\{\log(em)/m\}^{1/2}$ & $\{\log(en)/n\}^{\beta/(2\beta+d_x)}$\\[1mm]
		$d_y>2$, $\tau=0$ & $m^{-1/d_y}$ & $n^{-\beta/(\beta d_y+d_x)}$\\[1mm]
		$d_y>2$, $\tau>0$ & $m^{-(\tau+1)/(2\tau+d_y)}$ & $n^{-\beta(\tau+1)/[\beta(2\tau+d_y)+d_x(\tau+1)]}$\\
		\bottomrule
	\end{tabular}
\end{table}

Table~\ref{tab:phase-diagram} collects the unconditional and conditional rates. Positive Besov smoothness does not change the leading order when $d_y\le2$; for $d_y>2$ it improves the polynomial rate through the usual gain of one derivative under Wasserstein loss.

Figure~\ref{fig:dimension-phase} illustrates the conditional grid estimator with $X$ uniform on $[0,1]$ and conditional densities $q_x(y)=1+0.8\sin(2\pi x)(2y_1-1)$ on $[0,1]^d$, where $d=d_y\in\{1,2,4\}$. Panel (a) shows the integrated conditional estimation error, which includes localization bias. Panel (b) instead compares the empirical cell distribution with its pooled population law, isolating the sampling error in \eqref{eq:balancing-main}. If $K_d(n)$ is the number of covariate cells, the expected local sample count is $m_{n,d}=n/K_d(n)$. After multiplication by $\sqrt{m_{n,d}}$, the two-dimensional sampling curve increases slowly, while the one-dimensional curve remains approximately constant. These fixed-law simulations do not establish a minimax logarithmic exponent; Section~\ref{sec:response} proves the critical lower bound.

\begin{figure}[tp]
 \centering
 \includegraphics[width=0.80\textwidth]{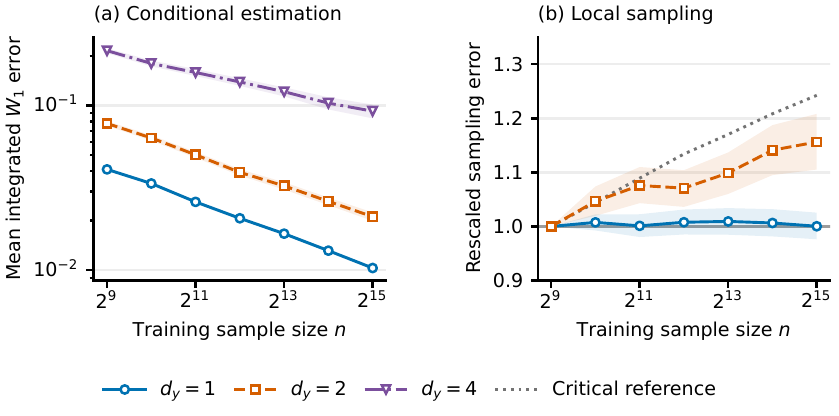}
 \small
 \caption{Conditional distribution estimation for $q_x(y)=1+0.8\sin(2\pi x)(2y_1-1)$ with $X$ uniform on $[0,1]$. Panel (a) shows mean integrated $W_1$ error of the grid empirical estimator in response dimensions $d=d_y\in\{1,2,4\}$. Panel (b) shows the mean sampling loss $W_1(\widehat Q_I,Q_I)$ for $d=1,2$, where $Q_I$ is the population law pooled over the query cell. In panel (b), each curve is multiplied by $\sqrt{m_{n,d}}$, with $m_{n,d}=n/K_d(n)$, and normalized at $n=512$. The dotted reference is $\{\log(e m_{n,2})/\log(e m_{512,2})\}^{1/2}$, and the horizontal reference is one. Shaded bands are pointwise $95\%$ bootstrap intervals incorporating numerical transport bounds. \AppendixRef{subsec:supp-dimension} gives the bandwidths, repetition counts, and evaluation procedures.}
 \label{fig:dimension-phase}
\end{figure}

The localization argument is not specific to $W_1$. For the unconditional class associated with Assumption~\ref{ass:wp-class}, the minimax rate, with $\widehat Q$ based on $m$ i.i.d. responses, is
\begin{equation}\label{eq:response-minimax-wp}
	\inf_{\widehat Q}
	\sup_{Q\in\Qcal^{(p)}_{d_y,\tau,r,v}(\underline q,\overline q)}
	\bigl\{\E_{Q^{\otimes m}} W_p^\kappa(\widehat Q,Q)\bigr\}^{1/\kappa}
	\asymp a_{d_y,\tau}(m)
\end{equation}
for all sufficiently large $m$. \AppendixRef{supp:w2-analytic} proves \eqref{eq:response-minimax-wp} for every fixed finite $\kappa$ using a one-sided comparison with a negative Sobolev norm. \AppendixRef{supp:w1-upper} gives the corresponding moment bounds for $W_1$.

\begin{theorem}[Conditional $W_p$ minimax rates]\label{thm:main-wp}
	Consider the class defined by Assumption~\ref{ass:wp-class} with fixed model parameters. Fix $1\le\kappa<\infty$. Then, for all sufficiently large $n$,
	\begin{equation}\label{eq:main-rate-wp}
		\mathcal R^\star_{n,p,\kappa}
		\bigl(\Pcal^{(p)}_{\beta,\tau,r,v}(\underline q,\overline q)\bigr)
		\asymp A_n(\tau)^\kappa\vee n^{-1}.
	\end{equation}
\end{theorem}

In particular, $\kappa=1$ gives integrated $W_p$ risk of order $A_n(\tau)$, while $\kappa=p$ gives integrated $W_p^p$ risk of order $A_n(\tau)^p\vee n^{-1}$.

The proof uses the localization argument of \AppendixRef{app:conditionalization}. Mixtures preserve the density and Besov constraints, and convexity of $W_p^p$ controls the pooling bias. For the lower bound, Proposition~\ref{prop:response-hypercube}(ii) controls the $W_p$ distance from each positive density in the hypercube family to the uniform law, and $W_p\ge W_1$ preserves the separation.

\subsection{Conditional Brenier map estimation}\label{subsec:brenier-result}

For the map risk \eqref{eq:map-risk-main}, Assumption~\ref{ass:brenier-class} provides uniform curvature control of the conditional Brenier maps. Lemma~\ref{lem:brenier-stability} in \AppendixRef{app:brenier} shows that, uniformly over $x\in S_X$, it implies
$$
	\lambda_0I
	\preceq D^2\varphi_x\preceq\Lambda I
	\quad\text{for a.e. }z\in\Y,
	\qquad\lambda_0=(\overline q\Lambda^{d_y-1})^{-1},
	\qquad
	q_x(y)\ge\Lambda^{-d_y}
	\quad\text{for a.e. }y\in\Y.
$$
The curvature assumption therefore supplies a uniform density lower bound.

\begin{theorem}[Conditional Brenier map minimax rate]\label{thm:brenier}
	Consider the class defined by Assumption~\ref{ass:brenier-class} with fixed model parameters. Then, for all sufficiently large $n$,
	\begin{equation}\label{eq:brenier-main}
		\inf_{\widehat T}\sup_{P\in\Pcal^{\mathrm{Br}}_{\beta,\tau,r,v}(\overline q,\Lambda)}
		\mathcal R_n^{\mathrm{map}}(\widehat T;P)
		\asymp A_n(\tau)^2.
	\end{equation}
\end{theorem}

Brenier stability converts the $p=\kappa=2$ distributional rate to squared $L^2(\rho)$ map error. Since $d_x\ge1$ gives $\gamma(\tau)<1/2$, the additional $n^{-1}$ term is of lower order than $A_n(\tau)^2$. Section~\ref{sec:brenier-main} proves the matching lower bound using inverse potentials.

\section{Conditional localization with unrestricted covariate distributions}\label{sec:conditionalization-main}

The key step is to localize an unconditional distribution estimator and weight each local error by the probability of its covariate cell. This mass weighting removes the need for a lower mass condition on $P_X$ for ordinary integrated Wasserstein loss. We first give the fixed grid argument for $W_1$. For higher moments, the concavity argument for a fixed grid is no longer sufficient in every regime, and \AppendixRef{app:conditionalization} replaces it with an adaptive dyadic partition.

\subsection{Localization on a fixed grid}

Let $\Pi$ be a finite measurable partition of $\X$. For $B\in\Pi$, write
$$
	\pi_B=P_X(B),
	\qquad
	Q_B=P_{Y\mid X\in B}
$$
when $\pi_B>0$. On an occupied cell, estimate $Q_B$ from the responses whose covariates fall in $B$; if fewer than two observations fall in the cell, use a fixed admissible distribution instead. Introduce the capped rate function
$$
	r(t)=
	\begin{cases}
		1,&0\le t<1,\\
		a_{d_y,\tau}(t),&t\ge1,
	\end{cases}
$$
with $\tau=0$ for unrestricted responses.

The count in cell $B$ is binomial with mean $n\pi_B$. Concentration and monotonicity of the unconditional rate imply
\begin{equation}\label{eq:fixed-grid-cell-risk-main}
	\E W_1(\widehat Q_B,Q_B)
	\lesssim r(n\pi_B).
\end{equation}
For $x\in B\cap S_X$, convexity of $W_1$ under mixtures together with the conditional H\"older condition gives
\begin{equation}\label{eq:fixed-grid-bias-main}
	W_1(Q_B,Q_x)
	\le
	\int_B W_1(Q_u,Q_x)\,P_X(du\mid B)
	\lesssim \operatorname{diam}(B)^\beta.
\end{equation}
It follows that, for a regular partition $\Pi_h$ with $M\asymp h^{-d_x}$ cells of side length $h$,
\begin{equation}\label{eq:fixed-grid-oracle-main}
	\E\int W_1(\widehat Q_x,Q_x)\,P_X(dx)
	\lesssim
	h^\beta+\sum_{B\in\Pi_h}\pi_B r(n\pi_B).
\end{equation}
The corresponding partition inequality, including the treatment of small binomial cell counts, is proved in Proposition~\ref{prop:partition}.

To control the stochastic term, let $F(t)=t r(t)$. For each unconditional rate in \eqref{eq:response-root-intro}, $F$ is concave on $[0,\infty)$. In the power-law cases, $F(t)=t$ for $0\le t<1$ and $F(t)=t^{1-\alpha}$ for $t\ge1$, with $0<\alpha<1$. For $d_y=2$, the upper branch is $F(t)=\sqrt{t\log(et)}$, which is also concave. In each case, the branches meet continuously at $1$, and the right derivative there is no larger than the left derivative. Jensen's inequality gives
\begin{align}
	\sum_{B\in\Pi_h}\pi_B r(n\pi_B)
	&=\frac1n\sum_{B\in\Pi_h}F(n\pi_B)\notag\\
	&\le \frac{M}{n}F(n/M)
	= r(n/M).\label{eq:fixed-grid-aggregation-main}
\end{align}
Equations \eqref{eq:fixed-grid-oracle-main} and \eqref{eq:fixed-grid-aggregation-main} imply
\begin{equation}\label{eq:conditional-upper-simple-main}
	\sup_P
	\mathcal R_{n,1,1}(\widehat Q;P)
	\lesssim
	h^\beta+r(n/M),
\end{equation}
uniformly over all Borel $P_X$. At the balancing bandwidth, $n/M\asymp nh^{d_x}\to\infty$, so $r(n/M)\asymp a_{d_y,\tau}(nh^{d_x})$. Balancing the two terms gives the upper bound of order $A_n(\tau)$.

\subsection{Matching lower bound}

For the lower bound it is enough to take $P_X$ uniform. Fix $h$ and choose $M_x\asymp h^{-d_x}$ disjoint smooth covariate bumps with supports of diameter $O(h)$. Place an independent copy of the hypercube family from Proposition~\ref{prop:response-hypercube} in each bump, attenuating the perturbation by the bump height outside its plateau. With local sample size $m\asymp nh^{d_x}$, assume
$$
	a_{d_y,\tau}(m)\lesssim h^\beta.
$$
The attenuation estimate then enforces the conditional $W_1$ H\"older condition. The Hellinger bound for a change in one sign keeps the total variation distance between neighboring $n$-sample laws bounded away from one. Summing over the covariate blocks produces integrated $W_1$ separation of order $a_{d_y,\tau}(m)$. At the balancing bandwidth \eqref{eq:balancing-main}, this gives
$$
	\mathcal R_{n,1,1}^\star
	\gtrsim A_n(\tau).
$$
Proposition~\ref{prop:replication} gives the full replicated hypercube argument, completing the proof of Theorem~\ref{thm:main} for $\kappa=1$.

\subsection{Higher moments and rare covariate regions}

For fixed finite $\kappa$, the local $W_1^\kappa$ risk is of order $a_{d_y,\tau}(m)^\kappa$ in the two $W_1$ classes. Under Assumption~\ref{ass:wp-class}, the local $W_p^\kappa$ risk has the same order. The difficulty is therefore not the local estimation problem but the aggregation of these errors over cells with very different probabilities under $P_X$. The fixed grid partition inequality remains valid, but the Jensen argument used above is not sharp for all $\kappa$.

For $p>1$, cellwise mixtures preserve the density bounds and the Besov condition. If $x\in B$, joint convexity of $W_p^p$ and \eqref{eq:cond-holder} give
$$
	W_p(Q_B,Q_x)^p
	\le
	\int_B W_p(Q_u,Q_x)^p\,P_X(du\mid B)
	\lesssim \operatorname{diam}(B)^{\beta p},
$$
so the pooling bias is again of order $\operatorname{diam}(B)^\beta$. The corresponding statement for $W_1$ follows directly from mixture convexity.

The proof of Proposition~\ref{prop:adaptive-localization} uses sample splitting and an adaptive dyadic partition. The partition adapts to local sample occupancy, with the regularity parameters $\beta$ and $\tau$ fixed and known. The first subsample, of size $n_0\asymp n$, accepts a level-$j$ cell when its count is at least $T_j$. The geometrically growing thresholds in \AppendixRef{app:adaptive-localization} satisfy $a_{d_y,\tau}(T_j)\lesssim2^{-j\beta}$. A cell $B$ is misclassified if it is accepted despite $n_0P_X(B)<T_j/2$, or rejected despite $n_0P_X(B)>2T_j$. Chernoff bounds, weighted by cell probabilities and summed over the $2^{jd_x}$ cells at each level, give
$$
 \E\sum_{B\text{ misclassified}}P_X(B)
 \lesssim\frac1n\sum_{j\ge0}2^{jd_x}T_je^{-cT_j}
 \lesssim n^{-1}.
$$

The first rejected cells form a disjoint terminal partition, and estimation uses their accepted parents. A terminal cell is treated as bad when it or its parent is misclassified. Misclassified terminal cells are charged to their own masses. The terminal children assigned to any misclassified parent have total probability at most that parent's probability. Thus the expected mass of all bad terminal cells is $O(n^{-1})$. Since Wasserstein loss is bounded on the response cube, their risk contribution has the same order.

On the remaining terminal cells, independence of the second subsample and the local moment bounds control distribution estimation error, while the conditional H\"older condition controls pooling bias. Write $\mathcal L_{\mathrm{good}}$ for these cells and $h_L=2^{-j(L)}$ for their side lengths. Both error terms have conditional $\kappa$th moments of order at most $h_L^{\beta\kappa}$ given the first subsample. Correct rejection also gives $P_X(L)\lesssim T_{j(L)}/n$.

The thresholds in \AppendixRef{app:adaptive-localization} satisfy
$$
 2^{jd_x}T_j\asymp
 2^{j\beta\kappa_c(\tau)}(1+j)^{\1\{d_y=2\}}.
$$
Choose $j_*$ so that $2^{j_*d_x}T_{j_*}\asymp n$ and put $h_*=2^{-j_*}$. Then $h_*^\beta\asymp A_n(\tau)$. Good terminal cells finer than $h_*$ have total probability at most one, so their contribution is at most $Ch_*^{\beta\kappa}$. For the coarser cells,
$$
 \begin{aligned}
 \sum_{\substack{L\in\mathcal L_{\mathrm{good}}\\j(L)\le j_*}}
 P_X(L)h_L^{\beta\kappa}
 &\lesssim \frac1n
 \sum_{\substack{L\in\mathcal L_{\mathrm{good}}\\j(L)\le j_*}}
 |L|\,2^{j(L)(d_x-\beta\kappa)}T_{j(L)}\\
 &\le \frac1n\max_{0\le j\le j_*}
 2^{j(d_x-\beta\kappa)}T_j
 \lesssim n^{-1}+h_*^{\beta\kappa},
 \end{aligned}
$$
where $|L|=h_L^{d_x}$ and disjointness gives $\sum_L|L|\le1$. If $\kappa\le\kappa_c(\tau)$, the maximum is bounded by a constant times its value at $j_*$; if $\kappa>\kappa_c(\tau)$, exponential decay makes it bounded independently of $n$. At $\kappa=\kappa_c(\tau)$, the endpoint bound retains the logarithmic factor when $d_y=2$. Together with the expected bad-cell mass of order $n^{-1}$, this gives, uniformly over all Borel $P_X$,
\begin{equation}\label{eq:higher-kappa-upper-main}
	\sup_P\mathcal R_{n,p,\kappa}(\widehat Q;P)
	\lesssim A_n(\tau)^\kappa+n^{-1},
\end{equation}
with $p=1$ in the unrestricted and $W_1$ Besov classes.

The two terms arise from different lower bounds. Under the uniform covariate distribution, replication of the distribution family gives $A_n(\tau)^\kappa$. For $p>1$, the positive family remains inside the prescribed density bounds and $W_p\ge W_1$ preserves its separation. The second term comes from a two-point covariate law with an atom of mass $n^{-1}$ carrying two separated distribution alternatives. With probability bounded away from zero that atom is not observed, forcing risk of order $n^{-1}$. Section~\ref{app:adaptive-localization} gives both constructions.

\subsection{Numerical illustration: rare covariate events}\label{subsec:num-risk}

To illustrate the effect of rare covariates on higher moments in Theorem~\ref{thm:main}, we take $d_x=d_y=1$ and consider
\begin{equation}\label{eq:num-scalar-law}
 q_{x,s}(y)=1+\alpha sx(2y-1),
 \qquad x,y\in[0,1],\quad s\in\{-1,1\},\quad \alpha=0.65.
\end{equation}
The conditional densities are bounded between $0.35$ and $1.65$, and the corresponding laws are Lipschitz in $x$ under $W_1$. We compare uniform covariates with
\[
 P_{X,n}=(1-n^{-1})\delta_0+n^{-1}\delta_1.
\]
Only the covariate distribution varies with $n$, as permitted by the minimax supremum over unrestricted $P_X$. Here $\beta=1$, so $A_n=n^{-1/3}$ and $\kappa_c=3$. We report the integrated risks for $\kappa=1$ and $\kappa=4$.

We compare the grid, occupancy, and empirical nearest neighbor estimators with a conditional normalizing flow using a monotone rational-quadratic spline \cite{DurkanEtAl2019}. The flow is fitted by likelihood and evaluated under the same Wasserstein loss. We use $n=2^8,\ldots,2^{15}$ and average the risks over the two signs in \eqref{eq:num-scalar-law}. \AppendixRef{subsec:supp-scalar} gives the estimator definitions, tuning parameters, and simulation procedures.

\begin{figure}[tp]
 \centering
 \includegraphics[width=0.80\textwidth]{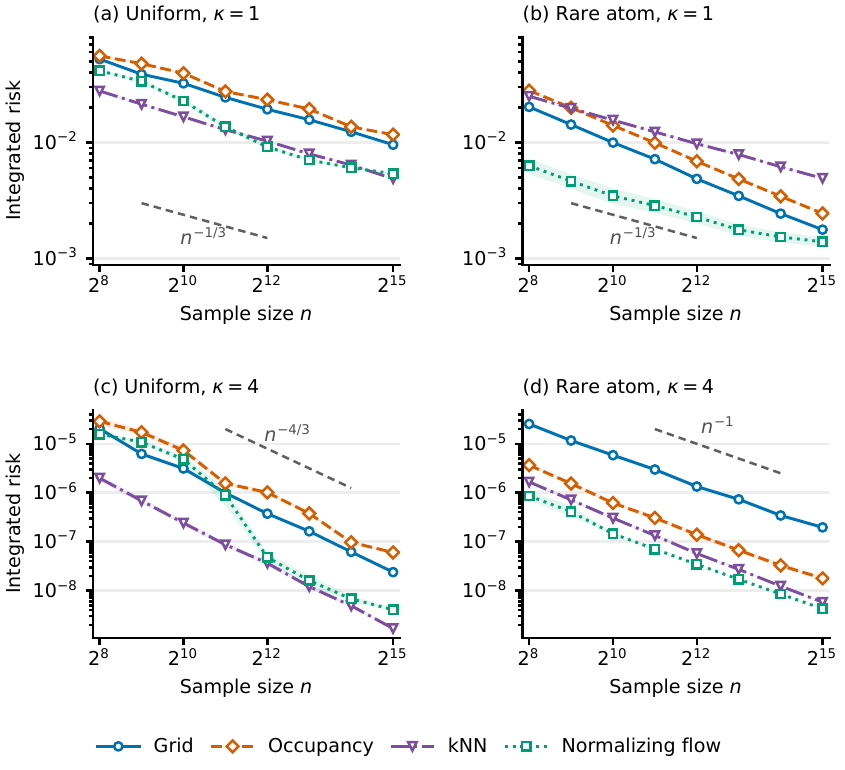}
 \small
 \caption{Conditional distribution learning for \eqref{eq:num-scalar-law}. The left and right columns use uniform and rare covariates, respectively. Panels (a) and (b) show mean integrated $W_1$ loss, and panels (c) and (d) show mean integrated $W_1^4$ loss, averaged over the two signs. All axes are logarithmic, with common vertical limits within each row. Gray dashed segments indicate reference powers with arbitrary vertical scaling. Shaded bands are pointwise $95\%$ bootstrap intervals over independent training datasets, with stratified resampling for rare covariates. \AppendixRef{subsec:supp-scalar} gives the repetition counts and integration procedures.}
 \label{fig:num-risk}
\end{figure}

Over the sample sizes in Figure~\ref{fig:num-risk}, the fourth-moment risks decrease more slowly under rare covariates than under uniform covariates. The gray segments show the first-moment benchmark $n^{-1/3}$ and, for the fourth moment, the uniform covariate benchmark $n^{-4/3}$ and the rare atom scale $n^{-1}$. These powers describe minimax benchmarks over the full model classes; individual estimators can behave differently on the fixed family in \eqref{eq:num-scalar-law}.

At $n=2^{15}$ under rare covariates, most fourth-moment error of the occupancy estimator arises when the rare atom is observed but only sparsely. The missed event establishes the lower bound, while the measured risk also includes this estimation error. This distinction also holds for the normalizing flow; \AppendixRef{subsec:supp-scalar} gives the risk decomposition and corresponding contributions.

\section{Unconditional Wasserstein estimation and multiscale lower bounds}\label{sec:response}

The conditional lower bound uses an unconditional family of distributions that can be attenuated across covariate neighborhoods without changing its rate. \AppendixRef{supp:w1-upper} gives the unconditional upper bounds and their extension to fixed finite moments, and \AppendixRef{app:response-lower} constructs the family used for the lower bound. A single multiscale potential construction gives the rates for $d=1$, $d=2$, and $d>2$.

\begin{theorem}[Unconditional $W_1$ minimax rates]\label{thm:response-benchmark}
	Let $d\ge1$ and fix a finite $1\le\kappa<\infty$. In the smooth case, also fix $\tau>0$ and $1\le r,v\le\infty$. Then, for all sufficiently large $m$,
	\begin{align}
		\inf_{\widehat Q}\sup_{Q\in\Qcal_d^{\mathrm{free}}}
		\E W_1^\kappa(\widehat Q,Q)
		&\asymp a_{d,0}(m)^\kappa,\label{eq:response-free-minimax}\\
		\inf_{\widehat Q}\sup_{Q\in\Qcal^B_{d,\tau,r,v}}
		\E W_1^\kappa(\widehat Q,Q)
		&\asymp a_{d,\tau}(m)^\kappa,
		\qquad \tau>0.\label{eq:response-minimax-w1}
	\end{align}
	Here and below, constants are allowed to depend on the fixed value of $\kappa$, but not on $m$.
\end{theorem}

For the upper bound, the empirical measure attains the unrestricted rate in every dimension and the smooth rate when $d\le2$. When $d>2$ and $\tau>0$, we truncate the empirical wavelet expansion at level $J$ and apply an approximate projection in the KR norm onto probability laws. The resulting estimator satisfies
\[
 \E W_1(\widehat Q_m,Q)
 \lesssim 2^{-J(\tau+1)}+m^{-1/2}2^{J(d/2-1)}.
\]
The first term is the approximation bias, with the gain of one derivative coming from the Lipschitz test functions in KR duality; the second is the stochastic error. Choosing $2^{J(2\tau+d)}\asymp m$ gives the rate $a_{d,\tau}(m)$. No pointwise density bounds are used. \AppendixRef{supp:w1-upper} proves these estimates and extends the upper bounds to every fixed finite moment.

The lower bound requires response laws that are difficult to distinguish statistically but well separated under Wasserstein loss. We combine perturbations across resolutions to accumulate separation while keeping neighboring alternatives difficult to distinguish. We use the multiscale potential
\begin{equation}\label{eq:critical-potential-family-main}
	U_\theta
	=\varepsilon\sum_{j=j_0}^J\sum_{k\in\mathcal K_j}
	\theta_{j,k}2^{-2j}\psi_{j,k},
	\qquad
	q_\theta=1-\Delta U_\theta,
	\qquad
	\varepsilon=cm^{-1/2}.
\end{equation}
Here $|\mathcal K_j|\asymp2^{jd}$ and the retained wavelets are supported in a fixed compact subset of $(0,1)^d$. The factor $2^{-2j}$ compensates for the two derivatives in the Laplacian. Thus the density perturbations have comparable testing cost across scales. At resolution $j$, each wavelet perturbation in the potential has gradient $L^2$ norm of order $\varepsilon2^{-j}$. With $\tau=0$ for the unrestricted branch, the upper resolution $J$ is chosen maximal subject to
$$
	\varepsilon2^{J(\tau+d/2)}\le c_0.
$$
\AppendixRef{app:response-lower} shows that this choice ensures positivity, the Besov constraint when $\tau>0$, and the one-coordinate testing bound. The dimensional transition is encoded by
$$
	B_J=\sum_{j=j_0}^J2^{j(d-2)}
	\asymp
	\begin{cases}
		1,&d=1,\\
		J,&d=2,\\
		2^{J(d-2)},&d>2.
	\end{cases}
$$
This sum determines the rate in each response dimension.

It remains to convert the energy of the multiscale potential into Wasserstein separation. The required estimate is a lower bound for the Kantorovich--Rubinstein norm.

\begin{proposition}[Potential lower bound for the KR norm]\label{prop:kr-lusin}
	Let $d\ge1$ and $U\in C_c^2((0,1)^d)$. Set $g=-\Delta U$. Then
	\begin{equation}\label{eq:kr-lusin-main}
		\|g\,dy\|_{\KR}
		\ge c_d\frac{\|\nabla U\|_{L^2}^{3}}{\|\nabla U\|_{L^4}^{2}},
	\end{equation}
	with the right-hand side understood to be zero if $g=0$.
\end{proposition}

The proof is based on quantitative Sobolev Lipschitz truncation \cite{AcerbiFusco1988,DieningMalekSteinhauer2008,Hajlasz1996}. A related argument for the empirical matching lower bound appears in \cite[Section~7]{BobkovLedoux2021}. Extend $U$ by zero and truncate it at gradient level $\lambda$. If $E_2=\|\nabla U\|_2^2$ and $E_4=\|\nabla U\|_4^4$, the truncation gives a Lipschitz function $\phi_\lambda$ satisfying
$$
	\int(-\Delta U)\phi_\lambda
	\ge E_2-CE_4\lambda^{-2},
	\qquad
	\operatorname{Lip}(\phi_\lambda)\lesssim\lambda.
$$
Choose $\lambda^2\asymp E_4/E_2$ and rescale the test function to have Lipschitz constant one. This establishes \eqref{eq:kr-lusin-main}; the full proof is given in \AppendixRef{app:response-lower}.

For two vertices $\theta,\theta'$, set $U=U_\theta-U_{\theta'}$ and let $H_j$ be the number of differing signs at resolution $j$. Define
\begin{equation}\label{eq:critical-S-main}
	S=\sum_{j=j_0}^J2^{-2j}H_j.
\end{equation}
The square function bounds give
$$
	\|\nabla U\|_2^2\gtrsim\varepsilon^2S,
	\qquad
	\|\nabla U\|_4^4\lesssim\varepsilon^4B_JS.
$$
Applying Proposition~\ref{prop:kr-lusin} now gives
\begin{equation}\label{eq:critical-separation-main}
	W_1(Q_\theta,Q_{\theta'})
	\gtrsim \varepsilon S/\sqrt{B_J}.
\end{equation}
Coordinatewise testing yields $\E S\gtrsim B_J$. Since $2^J\asymp m^{1/(2\tau+d)}$, with $\tau=0$ in the unrestricted branch,
\begin{equation}\label{eq:response-lower-cases-main}
	\varepsilon\sqrt{B_J}\asymp
	\begin{cases}
		m^{-1/2},&d=1,\\[1mm]
		\{\log(em)/m\}^{1/2},&d=2,\\[1mm]
		m^{-(\tau+1)/(2\tau+d)},&d>2,
	\end{cases}
\end{equation}
which is of order $a_{d,\tau}(m)$. The testing argument therefore proves the first-moment lower bounds in Theorem~\ref{thm:response-benchmark}. Jensen's inequality gives the corresponding lower bounds for every fixed finite $\kappa$. \AppendixRef{app:response-lower} also verifies the attenuation estimates needed for the conditional lower bound.

The critical case is $d=2$. Here each active resolution contributes at the same order to $B_J$, so $B_J\asymp J\asymp\log(em)$ and the accumulation across scales produces the factor $\sqrt{\log m}$. A family supported on one resolution yields only the parametric separation $m^{-1/2}$. Besov smoothness limits the finest scale but not this accumulation.

Figure~\ref{fig:num-critical} illustrates this accumulation using two-dimensional cosine densities, a numerical analogue of the interior-wavelet construction above. Combining levels gives greater mean transport separation than using the finest level alone at $J=2,3,4,5$ in panel (a). Panel (b) uses the corresponding $\sqrt J$ normalization. These five resolutions, with varying amplitudes, are insufficient to determine the asymptotic scaling. \AppendixRef{subsec:supp-critical} gives the construction and numerical transport bounds.

\begin{figure}[tp]
 \centering
 \includegraphics[width=0.80\textwidth]{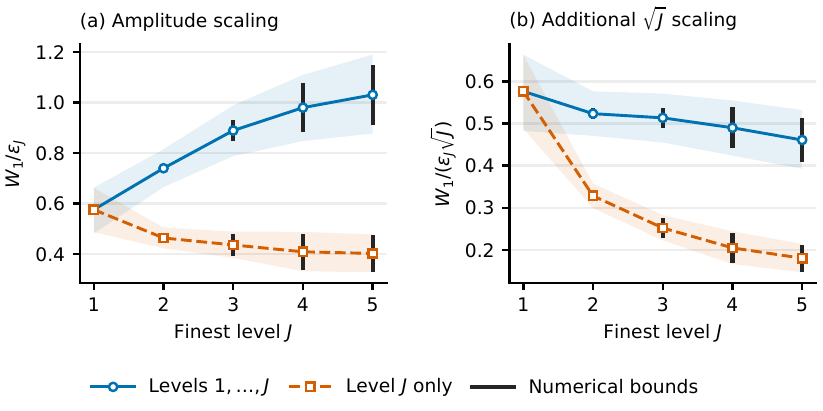}
 \small
 \caption{Transport separation for cosine densities using multiple levels or the finest level alone, with amplitude $\varepsilon_J=0.025\,2^{-2J}$. Panels (a) and (b) show $W_1/\varepsilon_J$ and $W_1/(\varepsilon_J\sqrt J)$. Markers and lines show means of numerical interval midpoints; black segments give mean lower and upper bounds. Shaded bands incorporate these bounds and pointwise $95\%$ bootstrap uncertainty over $32$ independent pairs of sign realizations at each resolution. \AppendixRef{subsec:supp-critical} gives the transport computations and numerical bounds.}
 \label{fig:num-critical}
\end{figure}

\section{Conditional Brenier map estimation}\label{sec:brenier-main}

Theorem~\ref{thm:brenier} follows from stability for the upper bound and an inverse potential construction for the lower bound.

\subsection{Stability and the minimax upper bound}

Lemma~\ref{lem:brenier-stability} specializes the fixed-source stability bound in \cite[Corollary~1]{BalakrishnanManole2026} to Assumption~\ref{ass:brenier-class}. Uniformly over $x\in S_X$, for every probability law $\widetilde Q$ on $\Y$,
\begin{equation}\label{eq:brenier-stability-main}
	\|T_{\rho\to\widetilde Q}-T_x\|_{L^2(\rho)}^2
	\le \frac{\Lambda}{\lambda_0}W_2^2(Q_x,\widetilde Q),
\end{equation}
where $\lambda_0=(\overline q\Lambda^{d_y-1})^{-1}$. The corresponding minimax comparison is made explicit in \cite{Balakrishnan2026}.

Apply Theorem~\ref{thm:main-wp} with $p=\kappa=2$ and density bounds $\underline q=\Lambda^{-d_y}$ and $\overline q$. Assumption~\ref{ass:brenier-class} and the density lower bound in Section~\ref{subsec:brenier-result} place the Brenier class inside the corresponding distribution class. There is therefore a measurable distribution estimator $\widehat Q$ such that, for the plugin map $\widehat T_x=T_{\rho\to\widehat Q_x}$, Lemma~\ref{lem:brenier-stability} gives
\begin{equation}\label{eq:brenier-plugin-upper}
 \sup_{P\in\Pcal^{\mathrm{Br}}_{\beta,\tau,r,v}(\overline q,\Lambda)}
 \mathcal R_n^{\mathrm{map}}(\widehat T;P)
 \lesssim A_n(\tau)^2+n^{-1}
 \asymp A_n(\tau)^2.
\end{equation}
The final comparison follows from $\gamma(\tau)<1/2$, since $d_x\ge1$. Only the true conditional map needs the curvature bounds in the stability lemma; the estimated law $\widehat Q_x$ need not lie in the Brenier class and may be discrete. \AppendixRef{app:brenier} records measurability of the plugin estimator. This proves the upper bound in Theorem~\ref{thm:brenier} for the full class.

\subsection{Lower bound from inverse potentials}

For the lower bound, we use the multiscale potentials $U_\theta$ and the spectral sum $B_J$ from Section~\ref{sec:response}, with $d=d_y$, but parameterize the inverse map. Define
$$
	S_U=\Id+\nabla U,
	\qquad
	T_U=S_U^{-1},
	\qquad
	q_U=\det(I+D^2U).
$$
The target density is then an explicit determinant, avoiding composition with an inverse map. For sufficiently small Hessian norm, \AppendixRef{app:inverse-potential-map} shows that $T_U$ is a Brenier map in the class of Assumption~\ref{ass:brenier-class}, the squared Hellinger distance between neighboring distributions is $O(\varepsilon^2)$, and
$$
	\|T_U-T_V\|_{L^2(\rho)}^2
	\asymp \|\nabla U-\nabla V\|_2^2.
$$
Embedding independent copies of this family in covariate neighborhoods of diameter $h$ gives the three constraints
\begin{equation}\label{eq:brenier-lower-three-constraints}
	nh^{d_x}\varepsilon^2\lesssim1,
	\qquad
	\varepsilon2^{J(\tau+d_y/2)}\lesssim1,
	\qquad
	\varepsilon\sqrt{B_J}\lesssim h^\beta.
\end{equation}
They correspond, respectively, to testing, smoothness in the response variable, and the conditional $W_2$ H\"older condition. Taking $\varepsilon\asymp(nh^{d_x})^{-1/2}$ and the largest admissible $J$ gives
$$
	\varepsilon\sqrt{B_J}
	\asymp a_{d_y,\tau}(nh^{d_x}).
$$
Balancing as in \eqref{eq:balancing-main} gives $h^\beta\asymp A_n(\tau)$ and squared $L^2(\rho)$ map separation $\varepsilon^2B_J\asymp A_n(\tau)^2$. \AppendixRef{app:inverse-potential-map} verifies the determinant identity, class membership, separation, and testing bounds, completing the lower bound in Theorem~\ref{thm:brenier}.

\section{Discussion}\label{sec:discussion}

The minimax results separate response distribution estimation from localization in the covariate. Weighting local errors by cell probabilities preserves the ordinary integrated Wasserstein rate $A_n(\tau)$ uniformly over unrestricted covariate distributions. Higher moments can instead be dominated by rare unobserved regions, which contribute the additional $n^{-1}$ term to the integrated $W_p^\kappa$ risk. This limitation applies to conditional generators even when the response laws are simple.

For comparison, a covariate density bounded below by a fixed $c_X\in(0,1]$ removes the rare region penalty. In the distribution classes considered here, the minimax integrated $W_p^\kappa$ risk is then of order $A_n(\tau)^\kappa$ for every fixed finite $\kappa$. Regular grid cells have probability at least $c_Xh^{d_x}$, so Proposition~\ref{prop:partition} gives the upper bound by balancing localization bias against error at sample scale $nh^{d_x}$. The lower bound under uniform covariates remains applicable.

Within the Brenier class, uniform curvature bounds convert squared $W_2$ estimation error into squared $L^2(\rho)$ map error of order $A_n(\tau)^2$. For smooth classes with $d_y>2$, the sharp upper bound uses a wavelet distribution estimator and an approximate metric projection; computational complexity bounds for this construction remain open. Attaining these minimax rates with computationally efficient parameterized generators remains a separate problem.

Further statistical questions concern adaptation to unknown $\beta$ and $\tau$ and extension to unbounded response spaces under suitable tail conditions. The density bounds for $p>1$ and the curvature assumptions for Brenier maps are sufficient conditions. It remains open how far these conditions can be relaxed without changing the rates.

\section*{Declarations}

The authors used OpenAI's GPT-5.6 and GPT-6 to assist with proof development, literature searches, numerical experiments, and manuscript editing. The authors directed the research, independently verified all AI-assisted work, and determined the final formulation and presentation of the results. The authors take full responsibility for all content.

\section*{Acknowledgments}

This work of Kim was supported by the National Research Foundation of Korea(NRF) grant funded by the Korea government(MSIT) (RS-2026-25598375).

\clearpage
\appendix
\numberwithin{figure}{section}
\renewcommand{\theHfigure}{appendix.\thesection.\arabic{figure}}
\section{Preliminaries}\label{supp:guide}
We fix the wavelet conventions used below and record the transport comparison and measurable projection used in the upper bounds. Implicit constants may depend on the fixed structural parameters and wavelet systems, but not on the sample size, resolution level, cell mass, or hypercube signs.

For each fixed $\tau\ge0$, choose an integer $K>\tau+2$ and fix a compactly supported tensor-product Cohen--Daubechies--Vial orthonormal boundary wavelet basis of $L^2([0,1]^d)$ whose regularity and number of vanishing moments exceed $K$. We use the basis without imposed boundary conditions: the scaling spaces reproduce affine functions, and every detail wavelet, including those at the boundary, has zero moments against affine functions; see \cite[Theorem~4.4]{CohenDaubechiesVial1993} and \cite[Section~2.12]{Cohen2003}. For the multiscale lower-bound construction, we retain only interior wavelets supported in a fixed compact subset of $(0,1)^d$; zero extension then permits the standard positive-order square function estimates.

Let $\Phi$ denote the finite scaling family and $\Psi_j$, $j\ge j_0$, the detail families on $[0,1]^d$. If $\alpha_f$ and $\beta_{f,j}$ are the associated coefficients, then the Besov norm is equivalent to
\begin{equation}\label{eq:besov-wavelet-supp}
	\|\alpha_f\|_{\ell_r}
	+\left\|\left(
	2^{j\{\tau+d(1/2-1/r)\}}\|\beta_{f,j}\|_{\ell_r}
	\right)_{j\ge j_0}\right\|_{\ell_v}.
\end{equation}
See \cite{CohenDaubechiesVial1993,Cohen2003,TriebelDomains2008}. We use this characterization throughout the proofs for distribution estimation.

Write $\Omega_d=[0,1]^d$. For a mean-zero function $u\in L^p(\Omega_d)$, $1<p<\infty$, set
\begin{equation}\label{eq:negative-def-wp}
	\|u\|_{-1,p}
	=
	\sup\left\{
	\left|\int_{\Omega_d}u\phi\right|:
	\phi\in W^{1,p'}(\Omega_d),\ \|\nabla\phi\|_{L^{p'}}\le1
	\right\},
	\qquad \frac1p+\frac1{p'}=1.
\end{equation}
Constants are immaterial because $u$ has zero mean.

\begin{lemma}[Flux representation]\label{lem:flux-negative-wp}
	For every mean-zero $u\in L^p(\Omega_d)$,
	$$
		\|u\|_{-1,p}
		=
		\min\left\{\|V\|_{L^p}:\nabla\!\cdot V=u\ \text{in distributions and }V\cdot n=0\ \text{on }\partial\Omega_d\right\}.
	$$
\end{lemma}

\begin{proof}
	Equip $W^{1,p'}(\Omega_d)/\mathbb R$ with the norm $\|\nabla\phi\|_{p'}$. The functional $[\phi]\mapsto-\int u\phi$ has norm $\|u\|_{-1,p}$. Hahn--Banach and $L^p$--$L^{p'}$ duality give a vector field $V$ with $\|V\|_p=\|u\|_{-1,p}$ and
	$$
		\int V\cdot\nabla\phi=-\int u\phi,
		\qquad \phi\in W^{1,p'}(\Omega_d).
	$$
	This is the weak formulation of $\nabla\!\cdot V=u$ with zero normal trace. The converse follows from H\"older's inequality.
\end{proof}

\begin{proposition}[One-sided transport comparison]\label{prop:wp-negative}
	Let $1<p<\infty$ and let $f,q\in L^p(\Omega_d)$ be probability densities with $q\ge\underline q>0$ almost everywhere. Then
	\begin{equation}\label{eq:wp-upper-one-sided}
		W_p(f\,dy,q\,dy)
		\le p\,\underline q^{-(p-1)/p}\|f-q\|_{-1,p}.
	\end{equation}
\end{proposition}

\begin{proof}
	Apply Lemma~\ref{lem:flux-negative-wp} to $u=f-q$ and choose $V$ with $\|V\|_p=\|f-q\|_{-1,p}$. Set
	$$
		\rho_t=(1-t^p)f+t^pq,
		\qquad
		j_t=pt^{p-1}V,
		\qquad 0\le t\le1.
	$$
	Then $\partial_t\rho_t+\nabla\!\cdot j_t=0$ and $\rho_t\ge t^pq$. The dynamic formulation of $W_p$ \cite{BenamouBrenier2000,Santambrogio2015} gives
	$$
		W_p^p(f\,dy,q\,dy)
		\le p^p\int_{\Omega_d}|V|^p q^{1-p}
		\le p^p\underline q^{1-p}\|V\|_p^p.
	$$
	Taking the $p$th root proves \eqref{eq:wp-upper-one-sided}. This is a one-sided comparison between $W_p$ and a negative Sobolev norm under a density lower bound; for the quadratic case see also \cite{Peyre2018}, and for finite Wasserstein orders see \cite{NilesWeedBerthet2022}.
\end{proof}

We also use the following measurable approximate projection. Let $d$ be induced by a norm on differences in an affine separable space, let $\mathcal C$ be the admissible set, and fix a countable $d$-dense family $\{f_1,f_2,\ldots\}\subset\mathcal C$. Given $\widetilde f$ and $\delta>0$, select the first index $i$ such that
$$
	d(\widetilde f,f_i)\le\inf_jd(\widetilde f,f_j)+\delta.
$$
Whenever the distances are measurable, this rule is measurable and, for every $f\in\mathcal C$,
\begin{equation}\label{eq:approx-projection}
	d(\widehat f,f)\le2d(\widetilde f,f)+\delta.
\end{equation}
We take $\delta=m^{-1}$. For the $W_1$ estimator, $d$ is the KR norm and rational finitely supported probability laws provide a countable dense family. For the $W_p$ estimator, $d$ is the Neumann negative norm $\|\cdot\|_{-1,p}$ on the separable set of $L^p$ probability densities. In each case the preliminary estimator depends on finitely many empirical wavelet coefficients, so its distance to each fixed candidate is Borel. The projected $W_p$ estimator need not satisfy the density bounds imposed on the true law; those bounds are used only in the one-sided comparison between Wasserstein distance and a negative Sobolev norm and the variance estimate. The projection defines a measurable statistical estimator.

\section{Unconditional upper bounds in \texorpdfstring{$W_1$}{W1}}\label{supp:w1-upper}
The unconditional upper bounds in Theorem~\ref{thm:response-benchmark} use the empirical measure for unrestricted responses and for smooth classes in dimensions one and two. In the smooth high-dimensional case, a truncated wavelet estimator suppresses fine-scale stochastic error and gains one derivative under the Kantorovich--Rubinstein metric.

\begin{proposition}[Unconditional upper bound]\label{prop:response-upper}
	Fix $1\le\kappa<\infty$. Let $m\ge1$ be an integer and let $Y_1,\ldots,Y_m$ be i.i.d. with law $Q$ on $[0,1]^d$. In the unrestricted model, there is an estimator satisfying
	$$
		\bigl\{\E W_1^\kappa(\widehat Q_m,Q)\bigr\}^{1/\kappa}\lesssim_\kappa a_{d,0}(m).
	$$
	If instead $Q\in\Qcal^B_{d,\tau,r,v}$, there is an estimator such that
	$$
		\bigl\{\E W_1^\kappa(\widehat Q_m,Q)\bigr\}^{1/\kappa}\lesssim_\kappa a_{d,\tau}(m).
	$$
	In dimensions $d\le2$, the empirical distribution itself attains this rate on the smooth class.
\end{proposition}

The smooth high-dimensional argument uses the following negative-order wavelet estimate.

\begin{lemma}[Endpoint wavelet upper bound]\label{lem:kr-wavelet-upper}
	Let $w$ be a mean-zero integrable function on $[0,1]^d$, with coefficients in the boundary wavelet basis $(\alpha_w,\beta_{w,\psi})$. Then
	\begin{equation}\label{eq:kr-wavelet-upper}
		\|w\,dy\|_{\KR}
		\le C\left\{
		\|\alpha_w\|_{\ell^1}
		+\sum_{j\ge j_0}2^{-j(1+d/2)}\|\beta_{w,j}\|_{\ell^1}
		\right\}.
	\end{equation}
\end{lemma}

\begin{proof}
	Because $w$ has mean zero, a $1$-Lipschitz test function $\varphi$ may be normalized by $\varphi(y_0)=0$. The finite scaling contribution is then uniformly bounded. For a detail wavelet $\psi$ at level $j$, choose $y_\psi\in\operatorname{supp}(\psi)$. Using $\int\psi=0$,
	$$
		|\langle\varphi,\psi\rangle|
		=\left|\int(\varphi(y)-\varphi(y_\psi))\psi(y)\,dy\right|
		\le C2^{-j}\|\psi\|_1
		\le C2^{-j(1+d/2)}.
	$$
	Applying this coefficient estimate to finite projections onto the boundary wavelet basis of $\varphi$ and passing to the uniform limit proves \eqref{eq:kr-wavelet-upper}.
\end{proof}

\begin{proof}[Proof of Proposition~\ref{prop:response-upper}]
	We first prove the assertions for $\kappa=1$.
	For the unrestricted branch, let $Q_m=m^{-1}\sum_{i=1}^m\delta_{Y_i}$. Standard empirical Wasserstein estimates \cite{FournierGuillin2015,WeedBach2019} give
	$$
		\E W_1(Q_m,Q)\lesssim a_{d,0}(m),
	$$
	where the two-dimensional case has the critical rate $\sqrt{\log(em)/m}$ by \cite[Theorem~3]{BobkovLedoux2021}. Since $a_{d,\tau}=a_{d,0}$ for $d\le2$, this also proves the smooth upper bound in low dimensions.

	Suppose now that $d>2$ and $\tau>0$. Let $\Phi$ and $\Psi_j$ be the boundary wavelet families fixed in \AppendixRef{supp:guide} and define the empirical projection
	$$
		\widetilde q_J
		=\sum_{\zeta\in\Phi}\widehat\alpha_\zeta\zeta
		+\sum_{j=j_0}^J\sum_{\psi\in\Psi_j}\widehat\beta_\psi\psi.
	$$
	The coefficients are empirical wavelet averages and hence unbiased for the coefficients of $q$. Write $q_J=\E\widetilde q_J$ for the population projection. Because constants belong to the scaling space, $\int\widetilde q_J=1$.

	The characterization by boundary wavelets implies
	$$
		\|\beta_{q,j}\|_{\ell^1}
		\le C2^{-j\tau}2^{jd/2},
	$$
	and Lemma~\ref{lem:kr-wavelet-upper} converts this into
	\begin{equation}\label{eq:kr-bias}
		\|(q-q_J)dy\|_{\KR}
		\le C2^{-J(\tau+1)}.
	\end{equation}
	For the variance term, finite overlap and $\int q=1$ give
	$$
		\sum_{\psi\in\Psi_j}\operatorname{Var}(\widehat\beta_\psi)
		\le\frac1m\int\sum_{\psi\in\Psi_j}|\psi(y)|^2q(y)\,dy
		\lesssim\frac{2^{jd}}m.
	$$
	Cauchy--Schwarz therefore yields
	$$
		\E\|\widehat\beta_j-\beta_{q,j}\|_{\ell^1}
		\le |\Psi_j|^{1/2}
		\left\{\sum_{\psi\in\Psi_j}\operatorname{Var}(\widehat\beta_\psi)\right\}^{1/2}
		\lesssim m^{-1/2}2^{jd}.
	$$
	No density upper bound is used in this calculation. Another application of Lemma~\ref{lem:kr-wavelet-upper} yields
	\begin{equation}\label{eq:kr-variance}
		\E\|\{\widetilde q_J-q_J\}dy\|_{\KR}
		\lesssim m^{-1/2}2^{J(d/2-1)}.
	\end{equation}
	Choose $J=J_m$ so that $2^{J(2\tau+d)}\asymp m$. Then \eqref{eq:kr-bias}--\eqref{eq:kr-variance} give
	$$
		\E\|(\widetilde q_J-q)dy\|_{\KR}
		\lesssim a_{d,\tau}(m).
	$$
	Applying the measurable approximate metric projection rule \eqref{eq:approx-projection} in the KR metric to the set of probability laws, with $\delta=m^{-1}$, gives a probability estimator satisfying
	$$
		W_1(\widehat Q_m,Q)
		\le2\|(\widetilde q_J-q)dy\|_{\KR}+m^{-1},
	$$
	establishing the smooth high-dimensional upper bound for sufficiently large $m$.

	We now extend these bounds to fixed finite moments. For the empirical estimators above, let $Z_m=W_1(Q_m,Q)$. Replacing one observation changes $Z_m$ by at most $\sqrt d/m$, by coupling the unchanged empirical atoms. For the smooth high-dimensional estimator, set $Z_m=\|(\widetilde q_J-q)dy\|_{\KR}$. Lemma~\ref{lem:kr-wavelet-upper} and finite overlap of the wavelet supports show that replacing one observation changes $Z_m$ by at most
$$
 \frac{C}{m}\left\{1+\sum_{j=j_0}^J
 2^{-j(1+d/2)}2^{jd/2}\right\}\le\frac{C'}{m}.
$$
The constants are independent of $J$. The bounded differences inequality \cite{McDiarmid1989} and integration of its upper tail therefore give, in each case,
$$
 \|Z_m\|_{L^\kappa(\mathbb P)}
 \le \E Z_m+C_\kappa m^{-1/2}
 \lesssim_\kappa a_{d,\tau}(m),
$$
where $\tau=0$ in the unrestricted model. The final inequality uses the preceding first-moment bounds and $a_{d,\tau}(m)\ge m^{-1/2}$. In the smooth high-dimensional case, \eqref{eq:approx-projection} transfers this bound to $W_1(\widehat Q_m,Q)$. Thus the same estimators satisfy
\begin{equation}\label{eq:response-w1-moments}
 \bigl\{\E W_1^\kappa(\widehat Q_m,Q)\bigr\}^{1/\kappa}
 \lesssim_\kappa a_{d,\tau}(m).
\end{equation}
In the smooth high-dimensional case, use a fixed probability law on the cube for the finitely many remaining positive sample sizes. Since $W_1$ is bounded by the cube diameter and $a_{d,\tau}(m)>0$, increasing the constant extends \eqref{eq:response-w1-moments} to every integer $m\ge1$, as required in Proposition~\ref{prop:partition}.
\end{proof}

\section{Unconditional lower bounds}\label{app:response-lower}
A single compactly supported potential hypercube is used in every response dimension. It supplies the testing, separation, and attenuation estimates required below, with the dimensional transition encoded by the spectral sum
\begin{equation}\label{eq:BJ-def}
	B_J=\sum_{j=j_0}^J2^{j(d-2)}.
\end{equation}
Thus $B_J\asymp1$ when $d=1$, $B_J\asymp J$ when $d=2$, and $B_J\asymp2^{J(d-2)}$ when $d>2$.

\subsection{Testing and metric decoding}
Write $a_m:=a_{d,\tau}(m)$. For probability densities $f,g$, we use $H^2(f\,dy,g\,dy)=\int(\sqrt f-\sqrt g)^2\,dy$ for the squared Hellinger distance. We combine a metric decoder with a hypercube of distributions. One-coordinate neighbors are statistically close, whereas accumulated coordinate errors yield $W_1$ separation of order $a_m$.

\begin{lemma}[Hypercube decoder]\label{lem:decoder}
	Let $(\mathcal M,d)$ be a metric space, let $\{P_\theta:\theta\in\{-1,1\}^M\}$ be a family of experiments, and associate to each vertex a point $f_\theta\in\mathcal M$. Assume that every single-coordinate flip satisfies
	$$
		\TV(P_\theta,P_{\theta^{(u)}})\le \delta_{\rm TV}<1.
	$$
	Given any estimator $\widehat f$, choose a measurable nearest vertex $\widehat\theta$, using deterministic tie breaking. Then
	$$
		d(f_\theta,f_{\widehat\theta})\le2d(f_\theta,\widehat f).
	$$
	Let $\pi$ be the uniform prior on the hypercube, and let $\E_\pi$ denote expectation over $\theta\sim\pi$ and the data drawn from $P_\theta$. Then, for every coordinate $u$,
	\begin{align}\label{eq:coordinate-error}
		\E_\pi\1\{\widehat\theta_u\ne\theta_u\}\ge\frac{1-\delta_{\rm TV}}{2}.
	\end{align}
\end{lemma}

\begin{proof}
	Nearest vertex optimality and the triangle inequality give the first claim. For the second, fix $u$ and pair every vertex with its $u$th-coordinate flip. The decision $\widehat\theta_u$ tests the two paired experiments, so the sum of its two error probabilities is at least $1-\TV(P_\theta,P_{\theta^{(u)}})$. Averaging over the pairs proves \eqref{eq:coordinate-error}.
\end{proof}

\subsection{Multiscale potential estimates}\label{app:potential-tools}
The same analytic argument applies in every dimension. We use quantitative Lipschitz truncation together with the standard positive-order wavelet square function equivalence.

\begin{proof}[Proof of Proposition~\ref{prop:kr-lusin}]
	Extend $U$ by zero to $\mathbb R^d$. Compact support gives $\int g=0$ and permits integration by parts against Lipschitz test functions.

	The case $U=0$ is immediate. Otherwise write $E_2=\int|\nabla U|^2$ and $E_4=\int|\nabla U|^4$. For any $\lambda>0$, apply \cite[Theorem~2.3]{DieningMalekSteinhauer2008} on the cube with an amplitude cutoff larger than $\|U\|_\infty$. After zero extension and a fixed rescaling of the truncation parameter, this gives a function $\phi$ and an exceptional set $E=\{U\ne\phi\}$ satisfying
	$$
		\operatorname{Lip}(\phi)\le\lambda,
		\qquad |E|\le C_dE_4\lambda^{-4},
		\qquad \nabla\phi=\nabla U\quad\text{a.e. on }E^c.
	$$
	The amplitude cutoff contributes no exceptional set. The bound for $|E|$ follows from $L^4$ boundedness of the Hardy--Littlewood maximal operator applied to $\nabla U$, and the gradient agreement follows from locality of weak gradients. H\"older's inequality gives
	$$
		\begin{aligned}
		\left|\int_E\nabla U\cdot(\nabla\phi-\nabla U)\right|
		&\le E_4^{1/2}|E|^{1/2}+\lambda E_4^{1/4}|E|^{3/4}\\
		&\le C E_4\lambda^{-2}.
		\end{aligned}
	$$
	Integration by parts therefore yields
	$$
		\int(-\Delta U)\phi
		\ge E_2-C E_4\lambda^{-2}.
	$$
	Choose $\lambda^2=C_0E_4/E_2$ with $C_0$ sufficiently large. Then $\int g\phi\ge E_2/2$, and $\phi/\lambda$ is an admissible test function for the KR norm. Hence
	$$
		\|(-\Delta U)dy\|_{\KR}
		\ge c_d\frac{E_2^{3/2}}{E_4^{1/2}},
	$$
	which proves \eqref{eq:kr-lusin-main}.
\end{proof}

\begin{lemma}[Positive-order square function for the interior family]\label{lem:positive-interior-wavelet}
	Let $1<s<\infty$ and
	$$
		V=\sum_{j=j_0}^J\sum_{k\in\mathcal K_j}c_{j,k}\psi_{j,k}.
	$$
	All retained $d$-dimensional wavelets are supported in a common compact cube $K_0\Subset(0,1)^d$. Then
	\begin{align}\label{eq:positive-interior-wavelet}
		\|\nabla V\|_{L^s([0,1]^d)}
		\asymp_s
		\left\|
		\left(
		\sum_{j,k}2^{2j}|c_{j,k}|^2|\psi_{j,k}|^2
		\right)^{1/2}
		\right\|_{L^s([0,1]^d)},
	\end{align}
	where the constants do not depend on $J$ or on the coefficient array.
\end{lemma}

\begin{proof}
	Let $\widetilde V$ be the zero extension of $V$ to $\mathbb R^d$. Every retained wavelet is an interior translate of a whole-space wavelet, so its whole-space coefficient is $c_{j,k}$ and all other coefficients, including the scaling coefficients, vanish. The order-one wavelet characterization gives
	\[
	 \|\widetilde V\|_{W^{1,s}(\mathbb R^d)}
	 \asymp_s
	 \left\|\left(\sum_{j,k}2^{2j}|c_{j,k}|^2
	 |\psi_{j,k}|^2\right)^{1/2}\right\|_{L^s(\mathbb R^d)},
	\]
	see \cite[Remark~3.8.5, equations~(3.8.42)--(3.8.43)]{Cohen2003}. The fixed basis has regularity greater than one and at least two vanishing moments, as required at order one.

	Poincar\'e's inequality bounds $\|V\|_{L^s}$ by $C_s\|\nabla V\|_{L^s}$, since $V$ vanishes near the boundary. The zero extension and the square function both vanish outside the response cube, giving \eqref{eq:positive-interior-wavelet}. The constants depend only on the fixed wavelet system, the exponent $s$, and the fixed support region; in particular, they are independent of $J$ and the coefficient array.
\end{proof}

The gradient estimates below give Wasserstein separation through Proposition~\ref{prop:kr-lusin}. The derivative estimates ensure positivity and Besov regularity, while the one-coordinate Hessian estimate controls the Hellinger distance between neighboring alternatives.

\begin{lemma}[Multiscale potential bounds]\label{lem:multiscale-potential}
	Let
	$$
		U_\theta
		=\varepsilon\sum_{j=j_0}^J\sum_{k\in\mathcal K_j}
		\theta_{j,k}2^{-2j}\psi_{j,k},
		\qquad \theta_{j,k}\in\{-1,1\},
	$$
	where $|\mathcal K_j|\asymp2^{jd}$ and the wavelets are the interior family fixed above. Then
	\begin{align}
		\|\nabla U_\theta\|_2^2&\lesssim\varepsilon^2B_J,\label{eq:potential-full-H1}\\
		\|D^{\ell+2}U_\theta\|_\infty&\lesssim_\ell\varepsilon2^{J(d/2+\ell)},\qquad \ell\in\{0,\ldots,K-2\}.\label{eq:potential-Ck}
	\end{align}
	If two vertices disagree in $H_j$ coordinates at resolution $j$ and
	$$
		S=\sum_{j=j_0}^J2^{-2j}H_j,
	$$
	then, for $V=U_\theta-U_{\theta'}$,
	\begin{align}
		\|\nabla V\|_2^2&\gtrsim\varepsilon^2S,\label{eq:potential-pair-H1}\\
		\|\nabla V\|_4^4&\lesssim\varepsilon^4B_JS.\label{eq:potential-pair-H14}
	\end{align}
	For a single-coordinate flip, one also has
	\begin{equation}\label{eq:potential-onebit-H2}
		\|D^2(U_\theta-U_{\theta^{(j,k)}})\|_2^2\lesssim\varepsilon^2.
	\end{equation}
\end{lemma}

\begin{proof}
	Lemma~\ref{lem:positive-interior-wavelet} with $s=2$ gives the first and third estimates. The derivative estimate \eqref{eq:potential-Ck} follows from finite overlap together with $\|D^{\ell+2}\psi_{j,k}\|_\infty\lesssim2^{j(d/2+\ell+2)}$.

	To prove \eqref{eq:potential-pair-H14}, write
	$$
		N_d(y)=\sum_{j=j_0}^J\sum_{k\in E_j}2^{-2j}|\psi_{j,k}(y)|^2,
	$$
	where $E_j$ denotes the set of differing coordinates. Lemma~\ref{lem:positive-interior-wavelet} gives $\|\nabla V\|_4^4\lesssim\varepsilon^4\int N_d^2$. Since $\int N_d=S$ by $L^2$ normalization and finite overlap implies $N_d\lesssim B_J$ pointwise, $\int N_d^2\lesssim B_JS$. For a one-coordinate flip, $\|D^2\psi_{j,k}\|_2\asymp2^{2j}$ together with the coefficient $2^{-2j}$ gives the Hessian estimate.
\end{proof}

For $q_\theta=1-\Delta U_\theta$, Proposition~\ref{prop:kr-lusin} and \eqref{eq:potential-pair-H1}--\eqref{eq:potential-pair-H14} give
\begin{equation}\label{eq:multiscale-response-separation}
	W_1(Q_\theta,Q_{\theta'})
	\ge c\frac{\varepsilon}{\sqrt{B_J}}
	\sum_{j=j_0}^J2^{-2j}H_j.
\end{equation}

\subsection{A unified hypercube family}\label{app:response-hypercube-unified}
\begin{proposition}[Least favorable hypercube with localization control]\label{prop:response-hypercube}
	Fix $d\ge1$, $\tau\ge0$, and $1\le r,v\le\infty$, and abbreviate $a_m=a_{d,\tau}(m)$. If $\tau>0$, use the fixed Besov slack from Assumption~\ref{ass:class}. For every sufficiently large $m$, there is a finite hypercube of distributions
	$$
		\{Q_\theta^{(m)}:\theta\in\{-1,1\}^{M_m}\},
	$$
	with densities $q_\theta^{(m)}$ having the following properties.
	\begin{enumerate}[label=(\roman*)]
		\item every density $q_\theta^{(m)}$ and every attenuated density $1+t(q_\theta^{(m)}-1)$, $0\le t\le1$, is bounded between fixed positive constants and can be made uniformly as close to the uniform density as required by choosing the fixed cutoff constant sufficiently small. When $\tau>0$, the entire attenuated family lies in the prescribed $B^\tau_{r,v}$ ball;
		\item with $Q_{\theta,t}$ denoting the attenuated law, for every fixed $1\le p<\infty$,
		$$
			W_p(Q_{\theta,t},Q_{\theta,t'})\le C_p a_m|t-t'|,
			\qquad
			W_p(Q_{\theta,t},\mathrm{Unif})\le C_p a_m.
		$$
		\item every one-coordinate flip satisfies
		$$
			H^2(Q_\theta^{(m)},Q_{\theta^{(u)}}^{(m)})\le C/m;
		$$
		\item if a decoder has coordinate error probability bounded below by a fixed $\eta>0$ on every coordinate, then
		$$
			\E_\pi W_1(Q_\theta^{(m)},Q_{\widehat\theta}^{(m)})\ge c_\eta a_m.
		$$
	\end{enumerate}
\end{proposition}

\begin{proof}
	Set $\varepsilon=cm^{-1/2}$ and use the potential family of Lemma~\ref{lem:multiscale-potential}. For $\tau\ge0$, with $\tau=0$ interpreted as the unrestricted branch, choose $J$ maximal subject to
	\begin{equation}\label{eq:potential-cutoff}
		\varepsilon2^{J(\tau+d/2)}\le c_0,
	\end{equation}
	where $c_0,c>0$ are fixed constants chosen below. Thus
	\begin{equation}\label{eq:potential-J-scale}
		2^J\asymp m^{1/(2\tau+d)}.
	\end{equation}
	Define $q_\theta=1-\Delta U_\theta$. Since every $U_\theta$ is supported in $K_0$, $\int q_\theta=1$. By \eqref{eq:potential-Ck},
	$$
		\|q_\theta-1\|_\infty\lesssim \varepsilon2^{Jd/2}.
	$$
	Thus \eqref{eq:potential-cutoff} ensures uniform lower and upper bounds on the densities.

	When $\tau>0$, choose an integer $\ell$ with $\tau<\ell\le K-2$. The same derivative estimate yields
	$$
		\|D^\ell(q_\theta-1)\|_\infty
		\lesssim_\ell
		\varepsilon2^{J(d/2+\ell)}.
	$$
	Because $q_\theta-1$ is supported a fixed positive distance from the boundary, its zero extension obeys
	$$
		\|\Delta_u^\ell(q_\theta-1)\|_{L^r}
		\lesssim
		\varepsilon2^{Jd/2}
		\min\{1,(2^J|u|)^\ell\},
		\qquad 1\le r\le\infty.
	$$
	The characterization by finite differences of $B^\tau_{r,v}$ now gives
	\begin{equation}\label{eq:potential-response-besov}
		\|q_\theta-1\|_{B^\tau_{r,v}}
		\lesssim
		\varepsilon2^{J(\tau+d/2)}.
	\end{equation}
	Multiplying $U_\theta$ by any $t\in[0,1]$ leaves these estimates unchanged up to constants, which proves part (i).

	Under a one-coordinate flip, the density difference is $-\Delta(U_\theta-U_{\theta^{(u)}})$. Its $L^2$ norm is $O(\varepsilon)$ by \eqref{eq:potential-onebit-H2}; uniform positivity then gives the Hellinger bound in part (iii). Choose $c_0$ sufficiently small for the prescribed density bounds and, when $\tau>0$, the Besov slack. With $c_0$ fixed, choose $c$ sufficiently small that
	$$
		\TV\bigl(Q_\theta^{\otimes m},Q_{\theta^{(u)}}^{\otimes m}\bigr)\le1/2.
	$$
	The constants in the preceding bounds are independent of $J$ and $\theta$, so these choices satisfy the density, regularity, and testing constraints simultaneously.

	For the localization estimate, integration by parts together with \eqref{eq:potential-full-H1} yields
	$$
		\|(q_\theta-1)dy\|_{\KR}
		\le\|\nabla U_\theta\|_1
		\le\|\nabla U_\theta\|_2
		\lesssim \varepsilon\sqrt{B_J}.
	$$
	Replacing $U_\theta$ by $(t-t')U_\theta$ proves part (ii) for $p=1$. For fixed $p>1$, all attenuated densities are bounded below by a common positive constant by part (i). Integration by parts gives
	$$
		\|-(t-t')\Delta U_\theta\|_{-1,p}
		\le |t-t'|\,\|\nabla U_\theta\|_{L^p}.
	$$
	Lemma~\ref{lem:positive-interior-wavelet} and finite overlap give $\|\nabla U_\theta\|_{L^p}\lesssim_p\varepsilon\sqrt{B_J}$: the square function in \eqref{eq:positive-interior-wavelet} is pointwise at most $C\varepsilon\sqrt{B_J}$. Proposition~\ref{prop:wp-negative} therefore gives
	$$
		W_p(Q_{\theta,t},Q_{\theta,t'})
		\lesssim_p |t-t'|\,\varepsilon\sqrt{B_J},
		\qquad
		W_p(Q_{\theta,t},\mathrm{Unif})
		\lesssim_p \varepsilon\sqrt{B_J}.
	$$
	Thus part (ii) holds for every fixed finite $p$. Moreover, \eqref{eq:BJ-def} and \eqref{eq:potential-J-scale} give
	\begin{equation}\label{eq:potential-radius-cases}
		\varepsilon\sqrt{B_J}\asymp
		\begin{cases}
			m^{-1/2},&d=1,\\[1mm]
			\{\log(em)/m\}^{1/2},&d=2,\\[1mm]
			m^{-(\tau+1)/(2\tau+d)},&d>2,
		\end{cases}
	\end{equation}
	which is of order $a_m$ by construction.

	For part (iv), let $H_j$ count coordinate errors at level $j$ and put $S=\sum_j2^{-2j}H_j$. The coordinatewise testing assumption implies
	$$
		\E_\pi S\gtrsim_\eta
		\sum_j2^{-2j}|\mathcal K_j|
		\asymp B_J.
	$$
	Equation~\eqref{eq:multiscale-response-separation} then gives
	$$
		\E_\pi W_1(Q_\theta,Q_{\widehat\theta})
		\gtrsim_\eta \varepsilon\sqrt{B_J}
		\asymp a_m,
	$$
	proving (iv).
\end{proof}

\subsection{Completion of the response theorem}

\begin{proof}[Proof of Theorem~\ref{thm:response-benchmark}]
	For the upper bound, Proposition~\ref{prop:response-upper} and \eqref{eq:response-w1-moments} give the asserted $\kappa$th-moment rate in both the unrestricted and smooth classes.

	For the minimax lower bound, use the hypercube in Proposition~\ref{prop:response-hypercube} and decode any estimator to a nearest vertex in $W_1$. The one-coordinate Hellinger estimate and Lemma~\ref{lem:decoder} ensure a uniformly positive coordinate error probability. Hence Proposition~\ref{prop:response-hypercube}(iv) gives
	$$
		\E_\pi W_1(Q_\theta,Q_{\widehat\theta})\ge ca_{d,\tau}(m).
	$$
	Nearest vertex optimality implies $W_1(Q_\theta,Q_{\widehat\theta})\le2W_1(Q_\theta,\widehat Q)$. Jensen's inequality therefore yields a Bayes lower bound of order $a_{d,\tau}(m)^\kappa$ under the uniform prior. Since minimax risk dominates this Bayes risk, the required lower bound follows.
\end{proof}

\section{Conditional localization with unrestricted covariate distributions}\label{app:conditionalization}
We prove the partition bound used in Section~\ref{sec:conditionalization-main}, then derive the lower bound under uniform covariates and the adaptive upper bound for higher moments. The argument applies to the two $W_1$ classes and, with the unconditional estimates of \AppendixRef{supp:w2-analytic}, to the $W_p$ density classes. A separate two-point construction gives the rare region lower bound.

\subsection{Partition bound weighted by cell probability}\label{app:partition-bound}
Let $\Pi$ be any finite measurable partition of $\X$. For each $B\in\Pi$, set
$$
	\pi_B=P_X(B),\qquad
	N_B=\sum_{i=1}^n\1\{X_i\in B\},
	\qquad
	Q_B=P_{Y\mid X\in B}\quad(\pi_B>0).
$$
For an unconditional rate $a=a_{d_y,\tau}$ and a fixed finite $1\le\kappa<\infty$, define the capped moment rate
\begin{align}\label{eq:capped-response-envelope}
	r_\kappa(t)=
	\begin{cases}
		1,&0\le t<1,\\
		a(t)^\kappa,&t\ge1.
	\end{cases}
\end{align}
Changing the value on a bounded interval affects only constants.

\begin{lemma}[Binomial transfer]\label{lem:binomial}
	Let $N\sim\mathrm{Bin}(n,\pi)$ and $\lambda=n\pi$. For every fixed finite $1\le\kappa<\infty$,
	\begin{align}\label{eq:bin-transfer}
		\E\!\left[a(N)^\kappa\1\{N\ge2\}+\1\{N<2\}\right]
		\le C_\kappa r_\kappa(\lambda),
	\end{align}
	where $a(N)$ on $N\ge2$ denotes any of the unconditional rates in \eqref{eq:response-root-intro} evaluated at the integer $N$.
\end{lemma}

\begin{proof}
	For bounded $\lambda$, the result follows after increasing $C_\kappa$. Assume henceforth that $\lambda\ge4$ and split according to $N\ge\lambda/2$. The explicit rates in \eqref{eq:response-root-intro} are nonincreasing and satisfy $a(\lambda/2)\le\sqrt2\,a(\lambda)$. Thus, on $\{N\ge\lambda/2\}$,
	$$
		a(N)^\kappa\le a(\lambda/2)^\kappa\le2^{\kappa/2}a(\lambda)^\kappa.
	$$
	On $\{N<\lambda/2\}$, the integrand in \eqref{eq:bin-transfer} is at most one, whereas Chernoff's inequality gives $\Pr(N<\lambda/2)\le e^{-c\lambda}$. Since $a(\lambda)\ge\lambda^{-1/2}$, for fixed finite $\kappa$ we have $e^{-c\lambda}\le C_\kappa\lambda^{-\kappa/2}\le C_\kappa a(\lambda)^\kappa$. Combining the two events proves \eqref{eq:bin-transfer}.
\end{proof}

\begin{proposition}[Partition bound weighted by cell probability]\label{prop:partition}
	Fix $1\le p<\infty$ and finite $1\le\kappa<\infty$. Assume the distribution class is closed under cellwise mixtures and that, from $m\ge2$ i.i.d. responses, one has an estimator satisfying
	$$
		\sup_Q\E W_p^\kappa(\widehat Q_m,Q)\le C_\kappa a_{d_y,\tau}(m)^\kappa.
	$$
	Apply this estimator on each cell with at least two observations, and use any fixed admissible fallback otherwise. Then
	\begin{equation}\label{eq:partition-bound}
		\begin{aligned}
		&\E\int W_p^\kappa(\widehat Q_x,Q_x)\,P_X(dx)\\
		&\quad\le C_\kappa\sum_{B:\pi_B>0}\left\{
		\int_B W_p^\kappa(Q_x,Q_B)\,P_X(dx)
		+\pi_B r_\kappa(n\pi_B)
		\right\}.
		\end{aligned}
	\end{equation}
\end{proposition}

\begin{proof}
	Conditional on $N_B=k$, the responses in cell $B$ are i.i.d. with distribution $Q_B$. The assumed bound for distribution estimation, followed by Lemma~\ref{lem:binomial}, therefore yields
	$$
		\E W_p^\kappa(\widehat Q_B,Q_B)
		\le C_\kappa r_\kappa(n\pi_B).
	$$
	For every $\kappa\ge1$,
	$$
		W_p^\kappa(\widehat Q_B,Q_x)
		\le2^{\kappa-1}\{W_p^\kappa(\widehat Q_B,Q_B)+W_p^\kappa(Q_B,Q_x)\}.
	$$
	Multiplying by $\pi_B$, integrating over the conditional distribution of $x$ in $B$, and summing over the partition yields \eqref{eq:partition-bound}. In the applications below, the required mixture closure is immediate for unrestricted laws; for smooth distribution classes it follows from convexity of the Besov constraint, together with preservation of any stated pointwise density bounds under mixtures.
\end{proof}

\subsection{Replication under the uniform covariate distribution}\label{app:replication}
\begin{proposition}[Replicated hypercube lower bound at $\kappa=1$]\label{prop:replication}
	Fix $\tau\ge0$ and $1\le r,v\le\infty$. Put $a_m=a_{d_y,\tau}(m)$ and take the hypercube family $\{Q_\theta^{(m)}\}$ from Proposition~\ref{prop:response-hypercube}. There exist constants $A_0,c_0>0$ such that, for sufficiently small $h$ and sufficiently large $m$,
	\begin{align}\label{eq:replication-m}
		m=\lfloor A_0nh^{d_x}\rfloor,
		\qquad
		a_m\le c_0h^\beta
	\end{align}
	imply
	\begin{align}\label{eq:replication-lower}
		\mathcal R^\star_{n,1,1}\ge c a_m
	\end{align}
	for the corresponding conditional model. This lower bound already holds when $P_X$ is uniform on $\X$.
\end{proposition}

\begin{proof}
	Choose $M_x\asymp h^{-d_x}$ smooth covariate bumps $b_a$ with supports of diameter $O(h)$ separated by distances of order $h$. Let $U_a$ be plateau regions with $|U_a|\asymp h^{d_x}$. For an array $\Theta=(\theta_1,\ldots,\theta_{M_x})$, set
	$$
		q_{\Theta,x}(y)
		=1+\sum_a b_a(x)(q_{\theta_a}^{(m)}(y)-1).
	$$
	The supports are disjoint, so at most one bump is active at any covariate value. We verify that the resulting conditional law has the required H\"older modulus. For a vertex $\theta$ and $t\in[0,1]$, write
	$$
		q_{\theta,t}=1+t(q_\theta^{(m)}-1),
		\qquad Q_{\theta,t}=q_{\theta,t}\,dy.
	$$
	Part (ii) of Proposition~\ref{prop:response-hypercube} yields
	$$
		W_1(Q_{\theta,t},Q_{\theta,t'})
		\le C a_m|t-t'|.
	$$
	The scaled bumps may be chosen so that
	$$
		|b_a(x)-b_a(x')|
		\le C\min\{1,|x-x'|/h\}
		\le C h^{-\beta}|x-x'|^\beta.
	$$
	If $x,x'$ lie in the same active bump, or if one point lies immediately outside it so that the corresponding attenuation is zero, the path estimate above and $a_m\le c_0h^\beta$ imply
	$$
		W_1(Q_{\Theta,x},Q_{\Theta,x'})
		\le Ca_mh^{-\beta}|x-x'|^\beta
		\le Cc_0|x-x'|^\beta.
	$$
	If the points belong to distinct active bumps, their supports are separated by at least $ch$. Since each conditional law is within $Ca_m$ of the uniform baseline by Proposition~\ref{prop:response-hypercube}(ii),
	$$
		W_1(Q_{\Theta,x},Q_{\Theta,x'})
		\le Ca_m\le Cc_0h^\beta\le C'c_0|x-x'|^\beta.
	$$
	Choose the fixed response-family constants as in Proposition~\ref{prop:response-hypercube}, so that every attenuated density remains in the prescribed class. With these constants fixed, choose $c_0$ sufficiently small that the H\"older constants in both cases are no larger than $L_X$.

	For the testing scale, flipping one coordinate in block $a$ changes the joint law of one observation by squared Hellinger distance at most $Ch^{d_x}/m$. With the preceding constants fixed, choose $A_0$ sufficiently large. Then $m=\lfloor A_0nh^{d_x}\rfloor$ and Hellinger tensorization keep adjacent $n$-sample total variation distances uniformly below one for sufficiently large $m$. For each $a$, decode an arbitrary estimator by nearest vertex in the integrated $W_1$ metric
	\begin{align}\label{eq:replication-first-metric}
		d_a(Q,Q')=\int_{U_a}W_1(Q_x,Q'_x)\,dx.
	\end{align}
	The testing lemma gives a uniformly positive coordinate error probability. Because the response vertex is constant on each plateau $U_a$, Proposition~\ref{prop:response-hypercube}(iv) gives
	$$
		\E_\pi d_a(Q_{\theta_a}^{(m)},Q_{\widehat\theta_a}^{(m)})
		\ge c|U_a|a_m.
	$$
	Nearest vertex optimality bounds the decoded distance by twice the estimator error integrated over $U_a$. Summing over $a$ and using $M_x|U_a|\asymp1$ gives a lower bound $ca_m$ for the integrated $W_1$ risk, proving \eqref{eq:replication-lower}.
\end{proof}

\subsection{Higher moments and adaptive localization}\label{app:adaptive-localization}

The fixed grid proof is useful for $\kappa=1$, but for larger $\kappa$ the function $t\mapsto t r_\kappa(t)$ need not be concave. The following adaptive construction removes this restriction without imposing regularity on $P_X$.

\begin{proposition}[Adaptive localization]\label{prop:adaptive-localization}
	Fix $1\le p<\infty$ and a finite $1\le\kappa<\infty$. For $p=1$, take either conditional distribution class in Theorem~\ref{thm:main}. For $p>1$, take the distribution class with densities bounded away from zero and infinity in Theorem~\ref{thm:main-wp}. There exists a conditional estimator satisfying
	\begin{equation}\label{eq:adaptive-upper}
		\sup_P\mathcal R_{n,p,\kappa}(\widehat Q;P)
		\lesssim A_n(\tau)^\kappa+n^{-1}
	\end{equation}
	for every integer $n\ge1$ and every Borel covariate distribution $P_X$.
\end{proposition}

\begin{proof}
	We first bound the expected mass of terminal cells affected by incorrect acceptance or rejection. We then control the risk on the remaining cells using the local moment bounds and disjointness of the terminal partition.

	It suffices to prove the bound for sufficiently large $n$. For the finitely many smaller positive sample sizes, use a fixed probability law for every query covariate; boundedness of $W_p^\kappa$ and positivity of $A_n(\tau)^\kappa+n^{-1}$ give the same upper bound after increasing the constant. Split the data into two subsamples of sizes $n_0$ and $n_1$ with $n_0\asymp n_1\asymp n$. Construct a dyadic partition of $\X$ from the first subsample. A level-$j$ cube has side length $h_j=2^{-j}$. Define
	$$
		T_j=\left\lceil C_0 2^{jb}(1+j)^\eta\right\rceil,
	$$
	where
	$$
		(b,\eta)=
		\begin{cases}
			(2\beta,0),&d_y=1,\\
			(2\beta,1),&d_y=2,\\
			\bigl(\beta(2\tau+d_y)/(\tau+1),0\bigr),&d_y>2.
		\end{cases}
	$$
	By construction, $a(T_j)\lesssim h_j^\beta$. In the $p>1$ case, use Proposition~\ref{prop:response-wp} for the local distribution estimator; cellwise mixtures preserve the density and Besov constraints, and joint convexity of $W_p^p$ gives the same $O(h_j^\beta)$ pooling bias. Declare a cell accepted when its count in the first subsample is at least $T_j$. For large $n$, the root is accepted, whereas $T_j>n$ at sufficiently fine levels; hence every dyadic chain has a deepest accepted cell. Use that cell for each query point. Acceptance is inherited by ancestors because counts decrease and thresholds increase down a chain. Fit the distribution estimator using only the second subsample in the selected cell, making the selected cell independent of the estimation sample.

	Let $\mathcal L$ be the disjoint partition formed by the first rejected cells. If $x\in L\in\mathcal L$, its pooling cell is the accepted parent $A(L)$; parent and child side lengths and thresholds differ only by constants, which we suppress. If a level-$j$ cell satisfies $n_0P_X(B)<T_j/2$, Chernoff's inequality bounds its acceptance probability by $e^{-cT_j}$. Multiplying by its mass, summing over the $O(2^{jd_x})$ cells at level $j$, and then over $j$ gives
	$$
		\frac Cn\sum_{j\ge0}2^{jd_x}T_je^{-cT_j}\lesssim n^{-1}.
	$$
	Likewise, when $n_0P_X(B)>2T_j$, the rejection probability is at most $e^{-cn_0P_X(B)}$. The bound $te^{-ct}\lesssim T_je^{-c'T_j}$ for $t>2T_j$ gives another total contribution of order $n^{-1}$. Let $\mathfrak B$ denote the sum of $P_X(B)$ over dyadic cells $B$ that are accepted in the first case or rejected in the second. The preceding bounds give $\E\mathfrak B\lesssim n^{-1}$. Define
	$$
		\mathcal L_{\mathrm{good}}
		=\{L\in\mathcal L:L\text{ and }A(L)\text{ are not misclassified}\}.
	$$
	For each accepted parent $B$, the terminal cells with $A(L)=B$ are disjoint subsets of $B$. Consequently, for every realization of the first subsample,
 $$
 \begin{aligned}
 \sum_{L\in\mathcal L\setminus\mathcal L_{\mathrm{good}}}P_X(L)
 &\le \sum_{\substack{L\in\mathcal L\\L\text{ misclassified}}}P_X(L)
 +\sum_{\substack{B\text{ accepted}\\B\text{ misclassified}}}
   \sum_{\substack{L\in\mathcal L\\A(L)=B}}P_X(L)\\
 &\le\mathfrak B.
 \end{aligned}
 $$
 Since $W_p^\kappa\le d_y^{\kappa/2}$ on the response cube, these terminal cells contribute at most $C\E\mathfrak B$ to the integrated risk. For $L\in\mathcal L_{\mathrm{good}}$, a terminal cell and its accepted parent satisfy
	\begin{equation}\label{eq:terminal-mass}
		P_X(L)\lesssim T_{j(L)}/n,
		\qquad
		n_0P_X(A(L))\gtrsim T_{j(L)}.
	\end{equation}
	The second inequality also implies $n_1P_X(A(L))\gtrsim T_{j(L)}$. Conditional on the first subsample, the count in the second subsample in $A(L)$ is binomial with this mean. Lemma~\ref{lem:binomial}, together with $a(T_{j(L)})\lesssim h_L^\beta$, bounds the distribution estimation term by $Ch_L^{\beta\kappa}$, including the fallback event. The conditional H\"older condition gives the same order for the pooling bias. Hence
	\begin{equation}\label{eq:terminal-risk}
		\mathcal R_{n,p,\kappa}(\widehat Q;P)
		\lesssim
		\E\sum_{L\in\mathcal L_{\mathrm{good}}}P_X(L)h_L^{\beta\kappa}
		+\E\mathfrak B.
	\end{equation}

	Choose $j_*$ so that
	\begin{equation}\label{eq:jstar-adaptive}
		2^{j_*(d_x+b)}(1+j_*)^\eta\asymp n,
		\qquad h_*=2^{-j_*}.
	\end{equation}
	At this scale $h_*^\beta\asymp A_n(\tau)$. Cells in $\mathcal L_{\mathrm{good}}$ finer than $h_*$ have total probability under $P_X$ at most one and hence contribute at most $h_*^{\beta\kappa}$. For the coarser cells in $\mathcal L_{\mathrm{good}}$, \eqref{eq:terminal-mass} gives
	$$
		\sum_{\substack{L\in\mathcal L_{\mathrm{good}}\\ j(L)\le j_*}}
		P_X(L)h_L^{\beta\kappa}
		\lesssim
		\frac1n\sum_{\substack{L\in\mathcal L_{\mathrm{good}}\\ j(L)\le j_*}}
		|L|h_L^{\beta\kappa-b-d_x}(1+j(L))^\eta.
	$$
	Here $|L|=h_L^{d_x}$ is the Euclidean volume of $L$. Disjointness gives $\sum_{L\in\mathcal L_{\mathrm{good}}}|L|\le1$, so the preceding display is at most
	\[
	 \frac Cn\max_{0\le j\le j_*}
	 2^{j(d_x+b-\beta\kappa)}(1+j)^\eta
	 \lesssim n^{-1}+h_*^{\beta\kappa}.
	\]
	If $d_x+b-\beta\kappa\ge0$, the maximum occurs at $j_*$ and \eqref{eq:jstar-adaptive} gives the second term. If $d_x+b-\beta\kappa<0$, the exponentially decreasing factor makes the maximum bounded independently of $n$, giving the first term. This includes the critical case without summing over resolutions. Together with $\E\mathfrak B\lesssim n^{-1}$ and $h_*^\beta\asymp A_n(\tau)$, substitution into \eqref{eq:terminal-risk} proves \eqref{eq:adaptive-upper}.
\end{proof}

\begin{proof}[Completion of the proofs of Theorems~\ref{thm:main} and~\ref{thm:main-wp}]
The two lower bounds are complementary. The uniform covariate distribution and the replicated family give
\begin{equation}\label{eq:higher-moment-uniform-lower}
	\mathcal R^\star_{n,p,\kappa}\gtrsim A_n(\tau)^\kappa,
\end{equation}
using Jensen's inequality together with the first-moment decoding bound. For $p>1$, the hard family lies inside the prescribed density bounds, Proposition~\ref{prop:response-hypercube}(ii) gives the required $W_p$ attenuation estimate, and $W_p\ge W_1$ preserves the separation.

For the second lower bound, take $n\ge2$, choose two covariate points $x_0,x_1$ a fixed distance apart and let
$$
	P_X=(1-n^{-1})\delta_{x_0}+n^{-1}\delta_{x_1}.
$$
Put the uniform distribution at $x_0$ and choose two fixed smooth perturbations of it, $Q_+,Q_-$, at $x_1$, with $\Delta=W_p(Q_+,Q_-)>0$ small enough to satisfy the conditional H\"older condition, any imposed Besov bound, and, when $p>1$, the stated density bounds. With probability $(1-n^{-1})^n\ge1/4$, no observation has covariate $x_1$, in which case the two data-generating laws coincide. For any estimated distribution $R$, the triangle inequality gives
$$
	\frac{W_p^\kappa(R,Q_+)+W_p^\kappa(R,Q_-)}2
	\ge (\Delta/2)^\kappa.
$$
Hence
\begin{equation}\label{eq:rare-design-lower}
	\mathcal R^\star_{n,p,\kappa}\gtrsim n^{-1}.
\end{equation}
Together, \eqref{eq:adaptive-upper} and \eqref{eq:rare-design-lower}, with the lower bound under the uniform covariate distribution above, establish Theorems~\ref{thm:main} and~\ref{thm:main-wp} for every fixed finite $\kappa$.\ 
\end{proof}

\section{Bounds for higher Wasserstein orders}\label{supp:w2-analytic}
We prove the $W_p$ bound used in Theorem~\ref{thm:main-wp}.

\subsection{\texorpdfstring{$W_p$}{Wp} distribution estimation}\label{app:wp-response}
We use the negative norm \eqref{eq:negative-def-wp} and the transport comparison in Proposition~\ref{prop:wp-negative}. Orthonormality across resolutions controls the empirical wavelet divergence field and gives the sharp logarithmic factor in dimension two.

\begin{proposition}[$W_p$ unconditional rates]\label{prop:response-wp}
	Consider $\Qcal^{(p)}_{d,\tau,r,v}(\underline q,\overline q)$ as defined in Section~\ref{subsec:regularity}, with fixed model parameters. Fix $1\le\kappa<\infty$. Then, for all sufficiently large $m$,
	\begin{equation}\label{eq:response-wp-supp}
		\inf_{\widehat Q}
		\sup_{Q\in\Qcal^{(p)}_{d,\tau,r,v}(\underline q,\overline q)}
		\bigl\{\E W_p^\kappa(\widehat Q,Q)\bigr\}^{1/\kappa}
		\asymp a_{d,\tau}(m).
	\end{equation}
\end{proposition}

\begin{proof}
	Let $Y_1,\ldots,Y_m$ be i.i.d. with density $q$, let $q_J$ be the population wavelet projection through level $J$, and let $\widetilde q_J$ be the corresponding estimator with empirical wavelet coefficients. The scaling space contains constants, so $\int\widetilde q_J=1$.

	Let $\widehat q_J$ be the measurable approximate projection of $\widetilde q_J$ onto the set of $L^p$ probability densities, using the rule of \AppendixRef{supp:guide} in the norm $\|\cdot\|_{-1,p}$ with tolerance $m^{-1}$, and write $\widehat Q_J=\widehat q_J\,dy$. Since the true density is admissible, \eqref{eq:approx-projection} and Proposition~\ref{prop:wp-negative} give
	$$
		\bigl\{\E W_p^\kappa(\widehat Q_J,Q)\bigr\}^{1/\kappa}
		\lesssim
		\|q-q_J\|_{-1,p}
		+
		\bigl\{\E\|\widetilde q_J-q_J\|_{-1,p}^\kappa\bigr\}^{1/\kappa}
		+m^{-1}.
	$$

	For each tensor-product wavelet $\psi\in\Psi_j$, choose a coordinate $i$ containing a detail factor and set
	$$
		V_\psi(z)=e_i\int_0^{z_i}
		\psi(z_1,\ldots,z_{i-1},s,z_{i+1},\ldots,z_d)\,ds.
	$$
	The detail factor has zero integral, so its primitive vanishes beyond its support. Thus $\operatorname{div}V_\psi=\psi$ and $V_\psi\cdot n=0$ on the boundary of the cube. These fields have support diameter $O(2^{-j})$, bounded overlap at each level, and $\|V_\psi\|_\infty\lesssim2^{j(d/2-1)}$. This is the wavelet field construction of \cite[Section~3.1, Lemmas~3 and~5]{NilesWeedBerthet2022}. Finite overlap gives
	$$
		\left\|\sum_{\psi\in\Psi_j}c_\psi V_\psi\right\|_{L^p}
		\lesssim 2^{-j}2^{jd(1/2-1/p)}\|c\|_{\ell^p}.
	$$

	The density upper bound implies $|\beta_{q,\psi}|\le\overline q\|\psi\|_1\lesssim2^{-jd/2}$, hence
	$\|\beta_{q,j}\|_{\ell^p}\lesssim2^{jd(1/p-1/2)}$. When $d>2$ and $\tau>0$, the Besov coefficient bound \eqref{eq:besov-wavelet-supp} and the class restriction $r\ge p$ instead give
	$\|\beta_{q,j}\|_{\ell^p}\lesssim2^{-j\tau}2^{jd(1/p-1/2)}$.
	The fields $\sum_{j>J}\sum_{\psi\in\Psi_j}\beta_{q,\psi}V_\psi$ therefore converge absolutely in $L^p$. Since the boundary wavelet projections converge to $q$ in $L^p$, integration by parts for the finite sums passes to every test function in $W^{1,p'}(\Omega_d)$. The limiting field has divergence $q-q_J$ and zero normal flux in this weak sense. Consequently,
	$$
		\|q-q_J\|_{-1,p}
		\lesssim\sum_{j>J}2^{-j}2^{jd(1/2-1/p)}\|\beta_{q,j}\|_{\ell^p},
	$$
	which yields
	\begin{equation}\label{eq:wp-bias}
		\|q-q_J\|_{-1,p}
		\lesssim
		\begin{cases}
			2^{-J},&d\le2\ \text{or}\ \tau=0,\\
			2^{-J(\tau+1)},&d>2,\ \tau>0.
		\end{cases}
	\end{equation}

	For the empirical detail error, define
	$$
		K_J(z,y)=\sum_{j=j_0}^J\sum_{\psi\in\Psi_j}V_\psi(z)\psi(y),
		\qquad
		G_J(z)=\frac1m\sum_{a=1}^m
		\{K_J(z,Y_a)-\E K_J(z,Y_a)\}.
	$$
	The divergence of $G_J$ is the detail part of $\widetilde q_J-q_J$, and $G_J$ has zero normal component. Write $B_J=\sum_{j=j_0}^J2^{j(d-2)}$. Orthonormality in $y$ and finite overlap give, for each $z$,
	$$
		\int_{\Omega_d}|K_J(z,y)|^2\,dy
		=\sum_{j=j_0}^J\sum_{\psi\in\Psi_j}|V_\psi(z)|^2
		\lesssim B_J.
	$$
	Also, $\sum_{j=j_0}^J\sum_{\psi\in\Psi_j}|\psi(y)|^2\lesssim2^{Jd}$, so Cauchy--Schwarz gives $|K_J(z,y)|^2\lesssim B_J2^{Jd}$. Set $t=\max\{2,p,\kappa\}$ and $Z_J(z,Y)=K_J(z,Y)-\E K_J(z,Y)$. Under $q\le\overline q$, these bounds imply
	$$
		\E|Z_J(z,Y)|^2\lesssim B_J,
		\qquad
		\E|Z_J(z,Y)|^t\lesssim_t B_J^{t/2}2^{Jd(t/2-1)}.
	$$
	Rosenthal's inequality \cite{Rosenthal1970}, applied componentwise to the independent observations, therefore gives
	$$
		\|G_J(z)\|_{L^t(\mathbb P)}
		\lesssim_{p,\kappa}
		\left(\frac{B_J}{m}\right)^{1/2}
		\left\{1+\left(\frac{2^{Jd}}m\right)^{1/2-1/t}\right\}
		\lesssim_{p,\kappa}\left(\frac{B_J}{m}\right)^{1/2},
		\qquad 2^{Jd}\le m.
	$$
	For $t=2$, the same bound follows from the variance identity. Since $t\ge\max\{p,\kappa\}$, Minkowski's inequality gives
	$$
		\bigl\|\|G_J\|_{L^p(dz)}\bigr\|_{L^\kappa(\mathbb P)}
		\le\bigl\|\|G_J\|_{L^t(\mathbb P)}\bigr\|_{L^p(dz)}
		\lesssim_{p,\kappa}\left(\frac{B_J}{m}\right)^{1/2}.
	$$

	The scaling part of $\widetilde q_J-q_J$ has mean zero because both projections have integral one and every detail wavelet has zero integral. Poincar\'e's inequality bounds its negative norm by its $L^p$ norm. The scaling family is fixed and finite, and its bounded functions have empirical coefficient errors of order $m^{-1/2}$ in every fixed finite moment. Combining this bound with the divergence field $G_J$ yields
	\begin{equation}\label{eq:wp-variance}
		\bigl\{\E\|\widetilde q_J-q_J\|_{-1,p}^\kappa\bigr\}^{1/\kappa}
		\lesssim_{p,\kappa}
		\left[\frac1m\left(1+\sum_{j=j_0}^J2^{j(d-2)}\right)\right]^{1/2}.
	\end{equation}
	For $d\le2$, take $J=\lfloor(\log_2m)/2\rfloor$. Substituting \eqref{eq:wp-bias} and \eqref{eq:wp-variance} into the preceding error decomposition gives $m^{-1/2}$ when $d=1$ and $\{\log(em)/m\}^{1/2}$ when $d=2$. For $d>2$, take $J=\lfloor(\log_2m)/(2\tau+d)\rfloor$, which balances the bias $2^{-J(\tau+1)}$ and stochastic error $m^{-1/2}2^{J(d-2)/2}$, including when $\tau=0$. Both choices satisfy $2^{Jd}\le m$ and yield the upper bound $C_{p,\kappa}a_{d,\tau}(m)$.

	For the finitely many positive sample sizes for which $J<j_0$, use a fixed probability law on the cube. Boundedness of $W_p$ then extends the moment upper bound to every integer $m\ge1$ after increasing the constant.

	For the lower bound, use the hypercube in Proposition~\ref{prop:response-hypercube}. Part (i) places the family in the prescribed density bounds and Besov ball after decreasing the fixed perturbation amplitude if necessary. Since $W_p\ge W_1$, the separation in part (iv) gives a first-moment Bayes lower bound of order $a_{d,\tau}(m)$. Jensen's inequality raises this to order $a_{d,\tau}(m)^\kappa$ for the $W_p^\kappa$ risk. This proves \eqref{eq:response-wp-supp} and hence the unconditional statement \eqref{eq:response-minimax-wp}.
\end{proof}

\section{Conditional Brenier map estimation}\label{app:brenier}
We establish the stability and measurability needed for the upper bound in Theorem~\ref{thm:brenier}, then prove the lower bound using inverse potentials.

\subsection{Stability and measurability}
\begin{lemma}[Uniform Brenier stability]\label{lem:brenier-stability}
	Fix $P\in\Pcal^{\mathrm{Br}}_{\beta,\tau,r,v}(\overline q,\Lambda)$ and write
	$$
		\lambda_0=(\overline q\Lambda^{d_y-1})^{-1}.
	$$
	Then, uniformly in $x\in S_X$,
	\begin{equation}\label{eq:brenier-density-floor}
		q_x(y)\ge\Lambda^{-d_y}
		\quad\text{for a.e. }y\in\Y,
	\end{equation}
	and for every probability law $\widetilde Q$ on $\Y$,
	\begin{equation}\label{eq:map-stability}
		\|T_{\rho\to\widetilde Q}-T_x\|_{L^2(\rho)}^2
		\le\frac{\Lambda}{\lambda_0}W_2^2(\widetilde Q,Q_x).
	\end{equation}
	In addition, every measurable map $S:\Y\to\Y$ satisfies
	\begin{equation}\label{eq:common-source}
		W_2^2(S_\#\rho,Q_x)\le\|S-T_x\|_{L^2(\rho)}^2.
	\end{equation}
\end{lemma}

\begin{proof}
	By the almost everywhere change of variables identity for the Brenier map \cite{Santambrogio2015}, $$
        q_x(T_x(z))\det DT_x(z)=1.
    $$
    Since $0\preceq DT_x\preceq\Lambda I$ and $q_x\le\overline q$, we have $\det DT_x\ge\overline q^{-1}$ almost everywhere, so the smallest eigenvalue of $DT_x$ is at least $\lambda_0$. Thus the Brenier potential is uniformly $\lambda_0$-strongly convex and $\Lambda$-smooth. The map $T_x$ is continuous, and its compact image has full $Q_x$-mass; because $q_x>0$ almost everywhere, that image must be the whole cube. Since the preceding bounds make $T_x$ bi-Lipschitz, $\det DT_x\le\Lambda^{d_y}$ in the same identity gives \eqref{eq:brenier-density-floor}.

	Equation~\eqref{eq:map-stability} now follows from \cite[Corollary~1]{BalakrishnanManole2026}, with fixed source $\rho$, strong convexity constant $\lambda_0$, and smoothness constant $\Lambda$. The estimated target $\widetilde Q$ may be any probability law on the cube. Finally, $(S,T_x)_\#\rho$ couples $S_\#\rho$ and $Q_x$, which proves \eqref{eq:common-source}.
\end{proof}

For the plugin estimators, consider the Brenier map operator
$$
 \mathsf B:\mathcal P(\Y)\to L^2(\rho;\mathbb R^{d_y}),
 \qquad \mathsf B(\nu)=T_{\rho\to\nu}.
$$
If $\nu_k\to\nu$ in $W_2$, qualitative stability of optimal maps \cite[Corollary~5.23]{Villani2009} gives $\mathsf B(\nu_k)\to\mathsf B(\nu)$ in $\rho$-probability, since $\rho$ is absolutely continuous. All maps take values in the cube, so the convergence also holds in $L^2(\rho)$. Thus $\mathsf B$ is continuous and Borel measurable. The distribution estimator in Theorem~\ref{thm:main-wp} is measurable by construction. It follows by composition that the plugin estimator $\widehat T_x=\mathsf B(\widehat Q_x)$ is measurable as an $L^2(\rho)$-valued function.

\subsection{Transport families from inverse potentials}\label{app:inverse-potential-map}
For the lower bound, we parameterize the inverse Brenier map directly. The target density is then an explicit determinant, avoiding composition with an inverse map.

\begin{lemma}[Brenier chart based on inverse potentials]\label{lem:inverse-brenier-chart}
	Let $U,V\in C_c^2((0,1)^d)$ and suppose
	$$
		\|D^2U\|_{\mathrm{op},\infty}
		\vee
		\|D^2V\|_{\mathrm{op},\infty}
		\le\eta<1.
	$$
	Define
	$$
		S_U=\Id+\nabla U,
		\qquad
		T_U=S_U^{-1},
		\qquad
		q_U=\det(I+D^2U).
	$$
	Then the following statements hold.
	\begin{enumerate}[label=(\roman*)]
		\item $S_U$ is a diffeomorphism of the cube onto itself, $T_U$ is the gradient of a convex potential, and
		$$
			(T_U)_\#\rho=q_U\,dy.
		$$
		Moreover
		$$
			(1-\eta)^d\le q_U\le(1+\eta)^d,
			\qquad
			(1+\eta)^{-1}I\preceq DT_U\preceq(1-\eta)^{-1}I.
		$$
		\item The corresponding target laws satisfy
		\begin{equation}\label{eq:inverse-chart-hellinger}
			H^2(q_U\,dy,q_V\,dy)
			\lesssim_\eta
			\|D^2(U-V)\|_2^2.
		\end{equation}
		\item The Brenier maps satisfy the two-sided comparison
		\begin{equation}\label{eq:inverse-chart-map}
			\|T_U-T_V\|_{L^2(\rho)}^2
			\asymp_\eta
			\|\nabla U-\nabla V\|_2^2.
		\end{equation}
	\end{enumerate}
\end{lemma}

\begin{proof}
	Extend $U$ by zero to $\mathbb R^d$. Then $\Psi_U(y)=|y|^2/2+U(y)$ is globally $(1-\eta)$-strongly convex and quadratic outside a compact set. For every $z\in\mathbb R^d$, the function $\Psi_U(y)-z\cdot y$ has a unique minimizer, so $S_U=\nabla\Psi_U$ is a bijection of $\mathbb R^d$. Since $DS_U=I+D^2U$ is positive definite, the local inverse theorem makes $S_U$ a global $C^1$ diffeomorphism. Since $S_U$ fixes every point outside the cube, injectivity and surjectivity imply that it maps the cube onto itself; see also \cite[Lemma~38]{HutterRigollet2021}. Its inverse is $T_U=\nabla\Psi_U^*$ and is therefore a Brenier map. Change of variables gives
	$$
		(S_U)_\#(q_U\,dy)=\rho,
	$$
	which is equivalent to $(T_U)_\#\rho=q_U\,dy$. The eigenvalue bounds on $I+D^2U$ immediately give the stated determinant and curvature bounds.

	The determinant is Lipschitz on the compact matrix ball $\{A:\|A\|_{\mathrm{op}}\le\eta\}$. Together with uniform positivity, this gives
	$$
		H^2(q_U\,dy,q_V\,dy)
		\lesssim\|q_U-q_V\|_2^2
		\lesssim\|D^2(U-V)\|_2^2,
	$$
	and hence \eqref{eq:inverse-chart-hellinger}.

	For \eqref{eq:inverse-chart-map}, change variables $z=S_U(y)$ and put $y'=S_V^{-1}(S_U(y))$. Then
	$$
		S_V(y)-S_V(y')=\nabla V(y)-\nabla U(y).
	$$
	Uniform bi-Lipschitz control of $S_V$ gives
	$$
		|y-y'|\asymp_\eta|\nabla U(y)-\nabla V(y)|.
	$$
	Since $q_U\asymp1$,
	$$
		\|T_U-T_V\|_{L^2(\rho)}^2
		=\int|y-y'|^2q_U(y)\,dy
		\asymp_\eta\|\nabla U-\nabla V\|_2^2.
	$$
\end{proof}

\begin{proof}[Proof of Theorem~\ref{thm:brenier} (lower bound)]
For the map lower bound, fix model parameters in the ranges of Assumption~\ref{ass:brenier-class} and take $P_X$ uniform on $[0,1]^{d_x}$. Use the covariate bumps from \AppendixRef{app:replication}: $M_x\asymp h^{-d_x}$ smooth functions $0\le b_a\le1$ with supports of diameter $O(h)$ separated by distances comparable to $h$, plateau regions of volume comparable to $h^{d_x}$, and H\"older moduli bounded by $Ch^{-\beta}$. For each $a$, let $\theta_a$ index an independent copy of the potential family in the response dimension in Lemma~\ref{lem:multiscale-potential}, and set
\begin{equation}\label{eq:map-inverse-family}
	U_{\Theta,x}(y)=\sum_a b_a(x)U_{\theta_a}(y),
	\qquad
	S_{\Theta,x}=\Id+\nabla U_{\Theta,x},
	\qquad
	T_{\Theta,x}=S_{\Theta,x}^{-1},
\end{equation}
with target density
\begin{equation}\label{eq:map-inverse-density}
	q_{\Theta,x}(y)
	=\det\{I+D_y^2U_{\Theta,x}(y)\}.
\end{equation}
Set
$$
	B_J=\sum_{j=j_0}^J2^{j(d_y-2)}.
$$

We verify membership in the Brenier class. By Lemma~\ref{lem:multiscale-potential},
$$
	\|D_y^2U_{\Theta,x}\|_\infty
	\lesssim \varepsilon2^{Jd_y/2}.
$$
If the fixed quantity $\varepsilon2^{Jd_y/2}$ is sufficiently small, Lemma~\ref{lem:inverse-brenier-chart} supplies the density and curvature bounds. When $\tau>0$, choose an integer $\ell$ with $\tau<\ell\le K-2$. Because the determinant is polynomial in the Hessian and the Hessians stay in a fixed small matrix ball,
$$
	\|q_{\Theta,x}-1\|_\infty\lesssim \varepsilon2^{Jd_y/2},
	\qquad
	\|D_y^\ell(q_{\Theta,x}-1)\|_\infty
	\lesssim_\ell \varepsilon2^{J(d_y/2+\ell)}.
$$
As the perturbation is supported in the fixed interior cube $K_0$, the characterization by finite differences implies
\begin{equation}\label{eq:map-inverse-besov}
	\|q_{\Theta,x}-1\|_{B^\tau_{r,v}}
	\lesssim
	\varepsilon2^{J(\tau+d_y/2)},
	\qquad 1\le r,v\le\infty.
\end{equation}
Thus, with the convention $\tau=0$ when no additional smoothness in the response variable is imposed, it is enough to impose
\begin{equation}\label{eq:map-inverse-smoothness-constraint}
	\varepsilon2^{J(\tau+d_y/2)}\le c_0.
\end{equation}

The conditional $W_2$ regularity follows from the coupling induced by the common source and \eqref{eq:inverse-chart-map}. Lemma~\ref{lem:multiscale-potential} gives $\|\nabla U_{\theta_a}\|_2\lesssim\varepsilon\sqrt{B_J}$. If $x,x'$ are in the same active covariate bump,
$$
	W_2(Q_{\Theta,x},Q_{\Theta,x'})
	\le\|T_{\Theta,x}-T_{\Theta,x'}\|_2
	\lesssim
	\varepsilon\sqrt{B_J}\,h^{-\beta}|x-x'|^\beta.
$$
The same estimate holds when, say, $x'$ lies outside all bump supports, since $b_a(x')=0$. If both points lie outside, the conditional laws coincide. For points in different active bumps, the supports are separated by a distance comparable to $h$, and both maps lie within $C\varepsilon\sqrt{B_J}$ of the identity. The conditional H\"older bound therefore follows provided
\begin{equation}\label{eq:map-inverse-holder-constraint}
	\varepsilon\sqrt{B_J}\le c_1h^\beta.
\end{equation}

Flipping one coordinate in block $a$ changes $D_y^2U_{\Theta,x}$ by an $L^2$ term of size $O(\varepsilon b_a(x))$. Lemma~\ref{lem:inverse-brenier-chart}(ii) consequently gives the one-sample testing estimate
\begin{equation}\label{eq:map-inverse-testing}
	H^2(P_\Theta,P_{\Theta^{(a,j,k)}})
	\lesssim h^{d_x}\varepsilon^2.
\end{equation}
On the plateau of covariate bump $a$, Lemma~\ref{lem:inverse-brenier-chart}(iii), together with \eqref{eq:potential-pair-H1}, gives the following separation for arbitrary vertices:
\begin{equation}\label{eq:map-inverse-separation}
	\int_\X\|T_{\Theta,x}-T_{\Theta',x}\|_{L^2(\rho)}^2\,dx
	\gtrsim
	h^{d_x}\varepsilon^2
	\sum_{a,j,k}2^{-2j}\1\{\theta_{a,j,k}\ne\theta'_{a,j,k}\}.
\end{equation}
Decode any map estimator using the metric
$$
	d(T,T')=\left\{\int_\X\|T_x-T'_x\|_{L^2(\rho)}^2\,dx\right\}^{1/2}.
$$
If $nh^{d_x}\varepsilon^2$ is bounded by a sufficiently small constant, Hellinger tensorization and Lemma~\ref{lem:decoder} give a uniformly positive coordinate error probability. Since $M_x\asymp h^{-d_x}$ covariate blocks are present, \eqref{eq:map-inverse-separation} implies
\begin{equation}\label{eq:map-inverse-risk}
	\inf_{\widehat T}\sup_\Theta
	\E\int_\X\|\widehat T_x-T_{\Theta,x}\|_{L^2(\rho)}^2\,dx
	\gtrsim \varepsilon^2B_J.
\end{equation}

Take
$$
	\varepsilon\asymp(nh^{d_x})^{-1/2}
$$
and choose $J$ as large as allowed by \eqref{eq:map-inverse-smoothness-constraint}. As in \eqref{eq:potential-radius-cases},
$$
	\varepsilon\sqrt{B_J}
	\asymp a_{d_y,\tau}(nh^{d_x}).
$$
Choosing $h$ according to \eqref{eq:balancing-main} also enforces \eqref{eq:map-inverse-holder-constraint}, and gives
$$
	\varepsilon^2B_J\asymp A_n(\tau)^2.
$$
The fixed construction constants can be chosen small enough that the density upper bound and the prescribed map curvature constant hold uniformly. This proves the map lower bound in every response dimension with the same family based on inverse potentials.
\end{proof}

\section{Details of the numerical experiments}\label{supp:numerical}
\ifcombinedpaper This appendix\else This section\fi{} gives the numerical procedures for Sections~\ref{sec:model-main}, \ref{sec:conditionalization-main}, and~\ref{sec:response}. For the conditional estimation and multiscale illustrations, the reported repetitions use seeds independent of the preliminary simulations. The experiment code and simulation records are not publicly available.

Pointwise $95\%$ percentile bootstrap intervals use $4{,}000$ resamples of independent repetitions, preserving paired signs, all queries within a dataset, and pairing between estimators evaluated on the same repetitions. Rare covariate simulations are resampled within count strata before recombination with exact probabilities.

For losses evaluated using numerical lower and upper bounds, plotted curves use mean interval midpoints. The shaded interval runs from the $2.5$th percentile of the bootstrap mean lower endpoints to the $97.5$th percentile of the bootstrap mean upper endpoints. These intervals combine approximate pointwise bootstrap uncertainty with numerical bracketing. The probabilistic transport lower bounds in Subsection~\ref{subsec:supp-dimension} have an additional failure probability $\delta$, separate from the bootstrap level. Quadrature and resolution comparisons are sensitivity checks rather than uniform error bounds.

\subsection{Conditional estimation across response dimensions}\label{subsec:supp-dimension}

For Figure~\ref{fig:dimension-phase}, $X$ is uniform on $[0,1]$, and the conditional response density on $[0,1]^d$ is
\[
 q_x(y)=1+\theta(x)(2y_1-1),\qquad \theta(x)=0.8\sin(2\pi x).
\]
The remaining response coordinates are independent uniforms. These densities satisfy $0.2\le q_x\le1.8$. Projection onto the first coordinate and a coupling with identical remaining coordinates give $W_1(Q_x,Q_{x'})=|\theta(x)-\theta(x')|/6$, so the conditional laws are Lipschitz with a constant independent of $d$.

For $n=2^9,\ldots,2^{15}$, the grid estimator partitions the covariate interval into $K_d(n)$ equal cells, where
\[
 K_1(n)=\lceil n^{1/3}\rceil,\qquad
 K_2(n)=\left\lceil\{n/\log(en)\}^{1/3}\right\rceil,\qquad
 K_4(n)=\lceil n^{1/5}\rceil.
\]
It uses the empirical response distribution in the query cell, with the uniform distribution as the empty-cell fallback. The bandwidth for $d=4$ uses the unrestricted response rate; this example does not assess the gain from response smoothness. We use $2{,}048$, $128$, and $24$ independent repetitions in dimensions $1$, $2$, and $4$, respectively.

Each repetition uses an independent uniform query. Conditional on its cell $I$, the retained count has distribution $\operatorname{Bin}(n,1/K_d(n))$, and the retained responses are independent with density $1+\bar\theta_I(2y_1-1)$, where $\bar\theta_I=|I|^{-1}\int_I\theta(x)\,dx$. We simulate this count and the retained responses directly, which has the same law as retaining observations from a full i.i.d. training sample. The cell parameter is evaluated analytically, and the first response coordinate is generated by inverting $F_I(t)=(1-\bar\theta_I)t+\bar\theta_I t^2$. The estimator itself uses only the retained responses. Averaging the losses over independent training samples and queries estimates the expected integrated risk. Common random variables couple the simulations across sample sizes, and bootstrap resampling preserves each complete trajectory. No empty query cell occurred in the reported repetitions.

Panel (a) evaluates losses against the population conditional law $Q_x$. Panel (b) uses the same empirical cell distributions but evaluates their sampling error against $Q_I$, the population law pooled over the query cell. The population pooling bias is $W_1(Q_I,Q_x)=|\bar\theta_I-\theta(x)|/6$. For $K$ equal cells, its integrated value is asymptotic to $2/(15K)$ and contributes to the total error in panel (a).

Scalar losses use analytic quantile integration. In dimensions two and four, partition the target cube into cells of side $1/M$ and compute their exact masses, centroids $c_C$, and within-cell variances $v_C$. A balanced transport linear program with source atoms $Y_i$ and costs $\sqrt{|Y_i-c_C|^2+v_C}$ gives an upper bound: independent lifting within each target cell and Jensen's inequality produce a feasible continuous transport plan, as in Subsection~\ref{subsec:supp-critical}.

For the lower bound, target dual values $b_C$ define $a_i=\min_C\{|Y_i-c_C|-b_C\}$. These values are $1$-Lipschitz on the source atoms, and $\phi(z)=\max_i\{a_i-|z-Y_i|\}$ is a globally $1$-Lipschitz extension. Thus $N^{-1}\sum_i a_i-\int\phi\,dQ_x$ is a lower bound when the cell contains $N$ observations. To evaluate the integral, partition the target cube into strata of side $1/H$, with exact probabilities $w_S=Q_x(S)$, and draw one independent observation $Z_S$ from each conditional law $Q_x(\cdot\mid S)$. The draws are independent of the training data and the computed potential. Subtracting
\[
 \left\{\frac{d\log(1/\delta)}{2H^2}\sum_S w_S^2\right\}^{1/2},
 \qquad \delta=10^{-7},
\]
from $N^{-1}\sum_i a_i-\sum_S w_S\phi(Z_S)$ gives a lower bound with failure probability at most $\delta$, by Hoeffding's inequality and the diameter $\sqrt d/H$ of each stratum. The same construction with $Q_I$ in place of $Q_x$ gives the sampling loss bounds in panel (b), using independent integration draws.

We use $(M,H)=(128,256)$ in dimension two and $(12,16)$ in dimension four. Primal and dual objectives are checked to agree within $10^{-7}$, and a rounding allowance of $10^{-9}$ is added to the upper bound and subtracted from the lower bound. Scalar losses are checked against independent numerical quadrature, and the continuous-target bounds are checked on examples with known transport costs. For panel (a) at $n=2^{15}$, the difference between the mean upper and lower bounds is $3.0\%$ and $15.7\%$ of the mean upper bound in dimensions two and four, respectively.

Panel (b) multiplies each mean sampling loss by $\sqrt{m_{n,d}}$, where $m_{n,d}=n/K_d(n)$, and normalizes it at $n=512$. Writing $\bar L_n^\ast$ and $\bar U_n^\ast$ for the bootstrap mean lower and upper endpoints, the normalized interval endpoints are
\[
 \sqrt{\frac{m_{n,d}}{m_{512,d}}}\,
 \frac{\bar L_n^\ast}{\bar U_{512}^\ast},
 \qquad
 \sqrt{\frac{m_{n,d}}{m_{512,d}}}\,
 \frac{\bar U_n^\ast}{\bar L_{512}^\ast}.
\]
Their $2.5$th and $97.5$th percentiles define the shaded band. At the normalization point, the curve and both endpoints are set to one. The reference curve is $\{\log(e m_{n,2})/\log(e m_{512,2})\}^{1/2}$. These fixed-law simulations do not determine the full conditional minimax logarithmic exponent.

\subsection{Scalar distribution estimation}\label{subsec:supp-scalar}

The grid estimator uses $\lceil n^{1/3}\rceil$ equal covariate cells and the empirical response distribution in each occupied cell. The empirical nearest neighbor estimator uses
\[
 k=\min\{n,\max(2,\lceil Cn^{2/3}\rceil)\},
 \qquad C=4.
\]

The occupancy estimator divides the sample into two halves. For each query, the first half selects the deepest dyadic cell at level $j$ such that every ancestor, at its respective level $l\le j$, contains at least $\lceil0.5\,2^{2l}\rceil$ observations. The empirical distribution of second half responses in the selected cell is then used for estimation. The empirical estimators use the uniform response distribution when an estimation cell is empty.

For the family in \eqref{eq:num-scalar-law}, the conditional distribution function is
\[
 F_{x,s}(y)=(1-\alpha sx)y+\alpha sxy^2.
\]
Responses are generated by inversion, with common uniform random variables for the two signs. For uniform covariates, we use $256$ independent repetitions per sample size, each evaluating both signs, with the nearest neighbor constant fixed beforehand. Empirical $W_1$ losses use analytic quantile integration, checked independently through the distribution functions.

Covariate integration uses exact probabilities for the two-point distribution. Under uniform covariates, grid and occupancy losses use $16$ Gauss points on each localization interval, and nearest neighbor losses use $2{,}048$ midpoints. Doubling these resolutions for the first $12$ datasets at each sample size changes the mean fourth-moment risk by at most $0.1945\%$ in the cases checked.

For the rare covariate design, let $K=\sum_{i=1}^n\1\{X_i=1\}$ count observations at the rare atom in the full sample, before sample splitting. For the empirical estimators, we simulate datasets conditional on $K=k$, for $k=0,\ldots,16$, placing the $k$ rare observations uniformly among sample indices. Each sample size uses $5{,}632$ conditional repetitions: $512$ for each $k=0,\ldots,8$ and $128$ for each $k=9,\ldots,16$. Conditional mean losses are weighted by exact $\operatorname{Bin}(n,1/n)$ probabilities. The omitted tail contribution is bounded by the maximal loss multiplied by $\Pr(K>16)$. The calculation is checked against $1{,}024$ independent samples from the original joint law per sample size.

The normalizing flow has eight spline intervals on $[0,1]$ and a uniform base distribution. A network with two hidden layers of 16 units and SiLU activations takes $2x-1$ as input and supplies the spline widths, heights, and derivatives, each bounded below by $10^{-3}$. The inverse spline generates responses.

Each fit starts from the uniform distribution. We use $80\%$ of the observations for likelihood fitting and the remaining $20\%$ for validation. Adam uses batches of 256 observations sampled with replacement, 1,600 updates, an $\ell_2$ penalty of $10^{-3}$, and a learning rate that decreases by a cosine schedule from $0.002$ to $2\times10^{-5}$. Gradient norms are clipped at 5. The initial model and checkpoints at every 40 updates are compared by validation likelihood. The two signs are fitted separately with paired initializations, data splits, and minibatch indices.

We estimate flow risks from 256 paired datasets per sample size under uniform covariates. Under rare covariates, we use 128 repetitions conditional on $K=k$ for each $k=0,\ldots,4$, 64 for each $k=5,\ldots,8$, and 16 for each $k=9,\ldots,16$, giving 1,024 stratified datasets per sample size. The weighting and tail treatment are the same as for the empirical estimators.

For flow losses, we integrate the absolute difference between fitted and true cumulative distribution functions using 64 Gauss nodes on each spline interval. Covariate integration uses 64 Gauss nodes under uniform covariates and exact masses for the two-point design. We double both resolutions for the first eight uniform datasets at each sample size and the first eight datasets in each rare stratum $K=0,1,2$. The largest relative change in a mean fourth-moment risk is $0.0003\%$ in these checks. When $K=0$, the two signs give identical training data and, under the paired training randomness, identical fitted flows.

Under uniform covariates, the sharp decrease in the flow's fourth-moment risk in Figure~\ref{fig:num-risk}(c) coincides with validation selecting the uniform initialization less often: it is selected in $53.1\%$ of fits at $n=256$, $2.5\%$ at $n=2{,}048$, and none at $n=4{,}096$.

For the two-point distribution, write $\ell_a=W_1(\widehat Q_a,Q_{a,s})$, $a\in\{0,1\}$. With $\E$ including the average over signs, the occupancy risk decomposes as
\begin{equation}\label{eq:num-risk-decomposition}
 \mathcal R_{n,1,4}
 =(1-n^{-1})\E[\ell_0^4]
 +n^{-1}\E[\ell_1^4\1\{K=0\}]
 +n^{-1}\E[\ell_1^4\1\{K\ge1\}].
\end{equation}
On $K=0$, the data cannot distinguish the two conditional laws at the rare atom. Their $W_1$ distance is $\alpha/3$, so the triangle inequality and convexity imply the lower bound
\begin{equation}\label{eq:num-missed-reference}
 B_n=\frac1n\left(1-\frac1n\right)^n\left(\frac\alpha6\right)^4
\end{equation}
for the second term in \eqref{eq:num-risk-decomposition}. An estimator using the uniform distribution at the unobserved atom attains this bound. An estimator that pools observations from a parent cell may incur a larger loss; accordingly, the dotted curve in Figure~\ref{fig:num-risk-decomposition} is a lower bound rather than the measured contribution.

\begin{figure}[tp]
 \centering
 \includegraphics[width=0.45\textwidth]{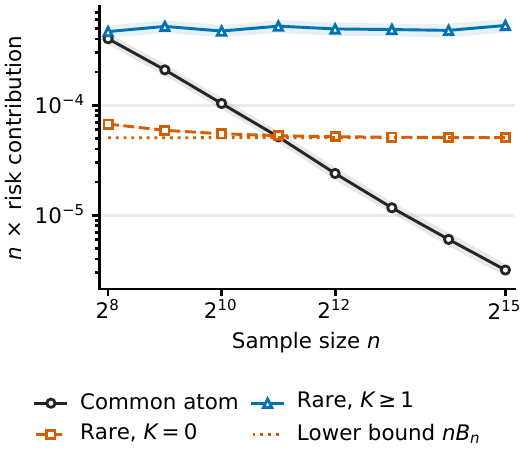}
 \small
 \caption{Fourth-moment risk of the occupancy estimator under rare covariates. The three contributions in \eqref{eq:num-risk-decomposition} are multiplied by $n$. Both axes are logarithmic. The dotted curve is the lower bound $nB_n$ from \eqref{eq:num-missed-reference}. Shaded bands are pointwise $95\%$ bootstrap intervals, using $5{,}632$ stratified repetitions at each sample size.}
 \label{fig:num-risk-decomposition}
\end{figure}

At $n=2^{15}$, the rare atom contributes $99.46\%$ and $98.87\%$ of the fourth-moment risks of the occupancy estimator and normalizing flow, respectively. The corresponding contributions from the unobserved rare atom are $8.69\%$ and $36.39\%$ of the total risks. The fitted flow laws at the rare atom remain close to uniform even when that atom is observed, so their risk also includes estimation error beyond the missed-event contribution.

\subsection{Multiscale transport separation}\label{subsec:supp-critical}

For the constructions in Figure~\ref{fig:num-critical}, let $\psi_k$ be the nonconstant orthonormal cosine modes on $[0,1]^2$, and define
\[
 I_j=\{k:2^{j-1}\le\max(k_1,k_2)<2^j\},
 \qquad \varepsilon_J=0.025\,2^{-2J}.
\]
For each $J\in\{1,2,3,4,5\}$, we generate $32$ independent pairs of sign vectors and form
\[
 q^{(a)}=1+\varepsilon_J\sum_{k\in I}\sigma_k^{(a)}\psi_k,
 \qquad a\in\{1,2\}.
\]
The finest level construction uses $I=I_J$, and the multiscale construction uses $I=\bigcup_{j=1}^J I_j$. The two constructions share signs on their common modes and coincide at $J=1$. The modes have zero mean, so the densities integrate to one. The cosine family permits exact cell averages in the transport calculations.

Both constructions satisfy $0.95\le q^{(a)}\le1.05$ and $\varepsilon_J2^j|I\cap I_j|^{1/2}\le0.025\sqrt3/2$ for $1\le j\le J$, giving positivity and the spectral block bound at $\tau=1$.

We compute exact cell averages on an $M^2$ grid, with $M=256$ for $J\le4$ and $M=512$ for $J=5$, and double $M$ for the first $8$ pairs. For $g=q^{(1)}-q^{(2)}$, let $z_C$ be the center of cell $C$ and set
\[
 g_M=\sum_C\left(\int_Cg(y)\,dy\right)\delta_{z_C}.
\]
Moving the signed mass in each cell to its center gives
\[
 \|g\,dy-g_M\|_{\KR}
 \le\frac{\sqrt2}{2M}\|g\|_1
 \le\frac{\sqrt2}{2M}\|g\|_2
 =\frac{\sqrt2}{2M}\|\delta\|_2=:\eta,
\]
where $\delta$ contains the differences in the orthonormal cosine coefficients, including $\varepsilon_J$.

Let $[L_G,U_G]$ be dual and primal bounds for transport of $g_M$ on the grid with eight neighbors. Horizontal and vertical edges have cost $1/M$, and diagonal edges have cost $\sqrt2/M$. The Euclidean distance $d_E$ and graph distance $d_G$ satisfy
\[
 d_E\le d_G\le\xi d_E,
 \qquad \xi=\sqrt{1+(\sqrt2-1)^2}.
\]
Indeed, for coordinate displacements $a\ge b\ge0$, the graph distance is $a+(\sqrt2-1)b$; Cauchy--Schwarz gives the upper comparison. Thus $L_G/\xi$ and $U_G$ bound the Euclidean transport cost of $g_M$.

We also compute bounds $[L_E,U_E]$ by grouping the positive and negative masses of $g_M$. After normalizing each part to a probability measure, we recursively split the group with largest mass-weighted spatial variance along its largest-variance coordinate at a weighted median, using at most $512$ groups for each measure. Each group is represented by its mass, centroid, and mean squared distance from the centroid. If $c_i,d_j$ are the centroids, $v_i,w_j$ the within-group variances, and $\pi_{ij}$ a feasible transport plan between group masses, independent lifting within each group pair and Jensen's inequality give the upper bound
\[
 \sum_{i,j}\pi_{ij}\sqrt{|c_i-d_j|^2+v_i+w_j}.
\]
For the lower bound, target dual values $b_j$ from the centroid transport problem define the globally $1$-Lipschitz function $\phi(z)=\min_j\{|z-d_j|-b_j\}$. We integrate $\phi$ against the original discrete measures. The bounds are multiplied by the common transported mass. Atoms of normalized mass at most $10^{-15}$ are omitted during grouping; their total mass is charged at the cube diameter, and residual and rounding allowances are added to the upper bound and subtracted from the lower bound.

Combining the discrete bounds and the signed discretization error yields
\begin{equation}\label{eq:num-critical-bracket}
 \max\{0,L_G/\xi-\eta,L_E-\eta\}
 \le W_1(q^{(1)}\,dy,q^{(2)}\,dy)
 \le \min\{U_G+\eta,U_E+\eta\}.
\end{equation}
Optimization uses a relative gap target of $1\%$ and a limit of $12{,}000$ iterations; the corrected bounds are retained even when the target gap is not achieved. Graph divergence residuals are corrected using a discrete Neumann Poisson solve, with any remaining residual charged at the graph diameter. Dual potentials are rescaled to satisfy the edge constraints. Centroid transport plans are accepted only when the maximum marginal residual is at most $10^{-8}$, with the resulting allowance included in the bounds. Costs, residuals, and feasibility are checked in double precision.

Independent checks against exact linear programming solutions cover $6$ random cases and $2$ multiscale cases. The largest relative graph primal--dual gap among the primary runs is $8.4\%$. In the doubled-grid checks, interval midpoints change by at most $2.4\%$, and all corresponding intervals overlap.

For $J=2,3,4,5$, the intervals for the mean difference in separation between the multiscale and finest level constructions, normalized by $\varepsilon_J$, lie strictly above zero after incorporating both numerical bounds and bootstrap uncertainty. At $J=1$, the multiscale and finest level constructions coincide for each paired sign realization, so the difference in their true separations is zero. The positive width of its numerical interval comes from subtracting the separately computed lower and upper bounds.

\FloatBarrier

\bibliographystyle{siamplain}
\bibliography{references}
\end{document}